\documentclass[a4paper,11pt]{amsart}
\usepackage{amsmath}
\usepackage{amssymb}
\usepackage{graphicx}
\usepackage{amscd}
\usepackage[all]{xy}
\usepackage{amsfonts}
\usepackage{mathrsfs}
\usepackage{ifthen}
\usepackage{amsthm}
\usepackage{leqno}

\newcommand\As{{\mathbb A}}

\newcommand\Vs{{\mathbb V}}

\newcommand\Ec{{\mathcal E}}
\newcommand\F{{\mathcal F}}
\newcommand\Hc{{\mathcal H}}
\newcommand\Ic{{\mathcal I}}
\newcommand\Jc{{\mathcal J}}
\newcommand\Lca{{\mathcal L}}

\newcommand\Oc{{\mathcal O}}
\newcommand\Pc{{\mathcal P}}
\newcommand\Sc{{\mathcal S}}

\newcommand\U{{\mathcal U}}
\newcommand\Spec{\operatorname{Spec}}
\newcommand\Spf{\operatorname{Spf}}

\newcommand\End{\operatorname{End}}
\newcommand\Endc{\operatorname{\mathcal End}}

\newcommand\Homc{\operatorname{\mathcal Hom}}

\newcommand\Det{\operatorname{Det}}

\newcommand\di{\operatorname{dim}}

\newcommand\Tg{\operatorname{T}}
\newcommand\Hcoh{\operatorname{H}}

\newcommand\Tr{\operatorname{Tr}}
\newcommand\expo{\operatorname{exp}}

\newcommand\ch{\operatorname{ch}}
\newcommand\Res{\operatorname{Res}}
\newcommand\iso{\kern.35em{\raise3pt\hbox
{$\sim$}\kern-1.1em\to}\kern.3em}
\newcommand\w{{\omega_X }}
\newcommand\Gr{\operatorname{Gr}}
\newcommand\Grc{\operatorname{\mathcal{G}r}}

\newcommand\Ide{\operatorname{Id}}

\newcommand\J{\operatorname{J}}
\newcommand\Higgs{{\Hc}iggs}

\theoremstyle{plain}
\newtheorem{thm}{Theorem}[section]

\newtheorem{lemma}[thm]{Lemma}
\newtheorem{prop}[thm]{Proposition}
\theoremstyle{definition}
\newtheorem{defin}[thm]{Definition}

\theoremstyle{remark}
\newtheorem{remark}[thm]{Remark}
\numberwithin{equation}{section}

\begin{document}

\title[Higgs pairs and Infinite Grassmannian]{Equations of the moduli of Higgs pairs and Infinite Grassmannian.}

\author[D. Hern\'andez-Serrano, J. M. Mu\~noz  and F. J. Plaza]{D. Hern\'andez-Serrano \\ J. M. Mu\~noz Porras \\  F. J. Plaza
Mart\'{\i}n}

\address{Departamento de Matem\'aticas, Universidad de
Salamanca,  Plaza
        de la Merced 1-4
        \\
        37008 Salamanca. Spain.
        \\
         Tel: +34 923294460. Fax: +34 923294583}

\thanks{
\indent This work is partially supported by the research contracts
MTM2006-0768 of DGI and  SA112A07 of JCyL. The first
author is also supported by MTM2006-04779. \\
\indent {\it E-mail addresses}: dani@usal.es, jmp@usal.es,
fplaza@usal.es}
\email{dani@usal.es}
\email{jmp@usal.es}
\email{fplaza@usal.es}

\begin{abstract}
In this paper the moduli space of Higgs pairs over a fixed smooth projective curve with extra formal data is defined and it is endowed with a scheme structure. We introduce a relative version of the Krichever map using a fibration of Sato Grassmannians and show that this map is injective. This fact and the characterization of the points of the image of the Krichever map allow us to prove that this moduli space is a closed subscheme of the product of the moduli of vector bundles (with formal extra data) and a formal anologue of the Hitchin base. This characterization also provide us the method to compute explicitely KP-type equations which describe the moduli space of Higgs pairs. Finally, for the case where the spectral cover is totally ramified at a fixed point of the curve, these equations are given in terms of the characteristic coefficients of the Higgs field.
\end{abstract}

\maketitle

\section{Introduction.}\qquad

In $1988$, Professor Hitchin studied (\cite{Hi}) the symplectic geometry of the cotangent space $\Tg^{\ast}\U_X$ to the moduli space $\U_X$ of vector bundles over a compact Riemann surface $X$, its points led him to introduce a new and nowadays relevant concept: the notion of Higgs pairs. From the point of view of the Sympletic Geometry, the map from $\Tg^{\ast}\U_X$ to an affine space of global sections (now called Hitchin map) turned out to be an algebraically completely integrable Hamiltonian system, fact that translates in Algebraic Geometry by saying that the fibers of the Hitchin map are Jacobians of a special curve covering $X$ (the spectral curve). As the author says in \cite{Hi}, one question is left: how to realize the Hamiltonian differential equations in some concrete way? 

Bearing in mind this question, Li and Mulase have partially solved the problem in \cite{LM1} using techniques of infinite integrable systems. They used the infinite Sato Grassmannian as a space of solutions for the KP system (\cite{Sa}) and have answered the question of how to recover the Hitchin system from the KP system in the case for which the spectral cover is not ramified. At the same period and with similar tools (Krichever morphism, Sato Grassmannian, ...), Donagi and Markman have also studied this topic in \cite{DM} with full generallity, they have shown certain compatibility for KP flows concerning spectral data coming from Higgs pairs on $X$. As far as we know, concrete equations describing this moduli space (and allowing ramification in the spectral cover) remains to be given. In this paper we try to solve this question. 

The paper is organized as follows. In section $2$ few words about the classical correspondence between Higgs pairs and line bundles over the spectral curve are said. Section $3$ is devoted to the study of the Krichever map for the moduli space of Higgs pairs with extra formal data, $\Higgs_X^{\infty}$. For this goal, an alternative method to that of \cite{LM1} and \cite{DM} is proposed: to use a fibration of infinite Sato Grassmannians (over a formal analogue $\As$ of the Hitchin base) and generalize the Krichever morphism. Theorem \ref{higgs:t:repHiggs} proves that $\Higgs_X^{\infty}$ is representable. Theorem \ref{higgs:t:KrIny} shows that the Krichever map is inyective. This fact, combined with the characterization of the image of the Krichever map (theorem \ref{higgs:t:carHiggs}) allow us to give the first main result: $\Higgs_X^{\infty}$ is a closed subscheme of $\U_X^{\infty}\times \As$ (theorem \ref{higgs:t:rep_en_U_X}, where $\U_X^{\infty}$ is the moduli scheme of vector bundles over $X$ with formal extra data). Section $4$ gives the spectral construction for the above topics defining a relative Grassmannian. Finally, in section $5$ the sencond main result is given, theorem \ref{eq:t:eqHiggs_X} provides concrete KP-type equations describing $\Higgs_X^{\infty}$ in terms of bilinear identities of Baker-Akhiezer functions and, for the connected component of $\Higgs_X^{\infty}$ on which the spectral cover is totally ramified, equations are explicitely computed in terms of the characteristic coefficients of the Higgs field (theorem \ref{eq:t:eq_tot_rami}).

\section{Preliminaries.}\label{higgs:s:pre}\quad

Denote $\U_X$ the moduli stack of rank $n$ vector bundles over a projective, smooth curve $X$ of genus $g$. Recall that a Higgs pair on $X$ consists of a vector bundle $E\in \U_X$ and a morphism of sheaves of $\Oc_X$-modules:
$$\varphi \colon E\to E\otimes \w\,.$$

%El fibrado $E$ est\'a dado (en t\'erminos de un recubrimiento abierto de $X$) por un 1-cociclo con valores en el haz %de grupos $\Gl(\Oc_X,n)$. Diferenciando este cociclo obtenemos un 1-cociclo con valores en el correspondiente haz de %\'algebras de Lie, luego se tiene
It is known (\cite{DM}):
$$\Tg^{\ast}_E \U_X \iso \Hcoh^0(X,\Endc_X E \otimes \w)\,.$$
and therefore, the cotangent space to $\U_X$ parametrizes Higgs pairs on $X$.

A remarkable fact is the existence of the Hitchin map:
\begin{align*}
H \colon \Tg^{\ast} \U &\to \oplus_{i=1}^n\Hcoh^0(X,\omega_X^i)\\
(E,\varphi) & \mapsto \ch(\varphi)
\end{align*}
where 
$$\ch(\varphi)=\big (\Tr(\varphi),\Tr(\Lambda^2 \varphi),\dots ,\Tr(\Lambda^n \varphi) \big )$$
are the characteristic coefficients of $\varphi$.

Another relevant issue is the spectral construction, whose importance lies on the fact that, for a generic  $s=(s_1,\dots ,s_n)\in \oplus_{i=1}^n\Hcoh^0(X,\omega_X^i)$, the points of $H^{-1}(s)$ corresponds to line bundles on the so called spectral curve $X_s$. We briefely describe its definition (\cite{Hi},\cite{BNR},\cite{Si}):
Consider the cotangent bundle:
$$\Tg^{\ast} X:=\Spec \Sc^{\bullet}\omega_X^{-1}$$
where $\Sc^{\bullet}\omega_X^{-1}$ is the sheaf of symmetric algebras  on $X$. Let $\Ic_s$ be the sheaf of ideals of $\Sc^{\bullet}\omega_X^{-1}$ generated by the image of:
\begin{align*}
\omega_X^{-n} & \to \Sc^{\bullet}\omega_X^{-1}=\oplus_{i\geq 0}\omega_X^{-1} \\
\xi & \mapsto (\xi \otimes s_n,\xi \otimes s_{n-1},\dots ,\xi \otimes s_1,\xi, 0, 0, \dots )
\end{align*}
The quotient $\Sc^{\bullet}\omega_X^{-1}/\Ic_s$ is a sheaf of $\Oc_X$-algebras and the spectral curve is defined as the zeros of $\Ic_s$ in $\Tg^{\ast}X$, that is:
$$X_s:=\Spec \Sc^{\bullet}\omega_X^{-1}/\Ic_s$$
The map:
$$\pi \colon X_s \to X$$
is a finite covering of degree $n$, where $n$ is the rank of $E$ and for generic $s$, $X_s$ is non singular. 

The bijection between $\J(X_s)$ and $H^{-1}(s)$
(where we denote the Jacobian of $X_s$ by $\J(X_s)$) is the following:

We can think of the Higgs field $\varphi$ as a morphism:
$$\omega_X^{-1}\to E \otimes E^{\ast}$$
equivalentely, as a morphism of $\Oc_X$-algebras
$$\Sc^{\bullet}\omega_X^{-1} \to E\otimes E^{\ast}\,.$$
Hamilton-Cayley theorem implies that this morphism factorizes by $\Oc_{X_s}=\Sc^{\bullet}\omega_X^{-1}/\Ic_s$ if and only if $\ch(\varphi)=s$. So given a point in $H^{-1}(s)$, that is, given $(E,\varphi)$ such that $\ch(\varphi)=s$, $E$ can be thought as a line bundle $L$ on $X_s$ (assuming smoothness) and $\pi_{\ast}L\iso E$. The converse is analogous.

\section{Moduli of Higgs pairs with extra formal data.}\quad

 We will work over the field of complex numbers and will denote it by $k$. We will assume as fixed, once and for all, the following data:
\begin{itemize}
\item $X$ smooth, projective, integral curve over $k$.
\item $x\in X$.
\item $t\colon \widehat \Oc_{X,x}\iso k[[z]]$ a formal trivialization of $X$ at $x$.
\end{itemize}
We want to consider a moduli problem parametrising:
 \begin{itemize}
 \item A rank $n$ vector bundle on $X$, $E$.
 \item A Higgs field $\varphi \colon E \to E \otimes \omega_X$.
 \item A formal trivialization $\phi$ of $E$ on $x$, i.e. $\phi \colon \widehat E_x \iso \widehat \Oc_{X,x}^{\oplus n}$
 \end{itemize}
We will see that certain compatibility between $\phi$ and $\varphi$ have to be imposed.
 \begin{defin}
 We call \emph{formal Higgs field} to the  morphism of sheaves of $\widehat\Oc_{X,x}$-modules:
 $$\widehat \varphi \colon \widehat E_x\to \widehat E_x\otimes \widehat \omega_{X,x}$$
 induced by the completion along $x$. The coefficients $a_i$ of the characteristic polynomial $p_{\widehat \varphi}(T)$ of $\widehat \varphi$  lie on $\Hcoh^0(\widehat X_x,\widehat \omega_{X,x}^{\otimes i} )$.
 \end{defin}

Since $X$ is smooh, the dualising sheaf $\w$ is a line bundle and formal trivialization $t$ induces a formal trivialization of $\w$ at $x$:
\begin{equation}\label{higgs:e:dt}
 dt \colon \widehat \omega_{X,x} \iso k[[z]]dz
\end{equation}
(we will forget about the generator $dz$). Then, the coefficients of the polynomial $p_{\widehat \varphi}(T)$ lies on  $k[[z]]$. Let us now introduce a $k$-scheme  parametrising these coefficients.

Let $\As^{\infty}$ be the infinite dimensional affine group scheme over $k$ with the addition group law, that is (\cite[Def. 4.13]{AMP}):
$$\As^{\infty}:=\Spec \varinjlim_l k[y_0, \dots ,y_l]=\Spec k[y_0,y_1,\dots ]$$
Let $R$ be a $k$-algebra. The correspondence:
$$s_0 +s_1 z+ \cdots \in R[[z]] \leftrightarrow (s_0,s_1,\dots )\in \As^{\infty}(R)$$
gives the following:
\begin{prop}
The functor, $\widetilde{k[[z]]}$, that associates to each $k$-algebra $R$ the ring of formal power series $R[[z]]$, is representable by $\As^{\infty}$.
\end{prop}
%\begin{corol}
%El functor
%$$\widetilde{k[[z]]^n} \colon R \rightsquigarrow R[[z]]^n$$
%es representable por el $k$-esquema $\As$ definido por
%$$\As=(\As^{\infty})^n=\Spec k[y_0^{(1)},\dots ,y_0^{(2)},\dots ,y_0^{(n)},\dots ]$$
%\end{corol}

\begin{prop}\label{higgs:p:As}
The functor on the category of $k$-schemes defined by:
$$S \rightsquigarrow \oplus_{i=1}^n \Hcoh^0(\widehat X_x\times S,\widehat{\omega}_{x\times S}^{\otimes i})$$
is representable by the $k$-scheme $\As=(\As^{\infty})^n$.
\end{prop}
%\begin{remark}
%We will denote $\As^{\infty}$ by:
%$$\As=\oplus_{i=1}^n \As_i$$
%\end{remark}

\begin{remark}\label{higgs:r:As}
We will think of the points of $\As$ with values in $S$ as polynomials
$$T^n+a_1T^{n-1}+ \cdots +a_n$$
where $a_i \in \Hcoh^0(\widehat X_x \times S,\widehat \omega_{x\times S}^{\otimes i})$.
\end{remark}

\begin{defin}\label{higgs:d:higgs}
 We define the functor $\Higgs_X^{\infty}$ as the sheafication of:
\begin{align*}
\mathcal C_{k-sch} & \rightsquigarrow \mathcal C_{sets}\\
 S &\mapsto \{(E,\phi,\varphi)\}/\sim
\end{align*}
where
\begin{enumerate}
\item $E$ is a rank $n$ vector bundle over $X\times S$.
\item $\varphi \colon E \to E \otimes \omega_{X\times S/S}$ is a relative Higgs field, that is, a morphism of sheaves of $\Oc_{X\times S}$-modules.
\item $\phi$ is a formal trivialization of $E$ along $x\times S$ 
$$\phi \colon \widehat E_{x\times S}\iso \widehat \Oc_{x\times S}^{\oplus n}$$
such that there exists a matrix 
$$M=\left( \begin{array}{cccccc}
0 & & & & 0 & (-1)^{n+1}a_n\\
1 & \cdot & & & \cdot & (-1)^na_{n-1} \\
\cdot & 1 &\cdot  & & \cdot & \cdot \\
\cdot & & \cdot & \cdot & \cdot & \cdot \\
\cdot & & & \cdot & 0 & -a_2 \\
0 & \cdot & & 0 & 1 & a_1
\end{array} \right)$$
(where $(a_1,\dots ,a_n)\in \As(S)$) making the following diagram commutative:
$$\xymatrix{
\widehat E_{x\times S} \ar[r]^{\widehat \varphi \qquad} \ar[d]^{\wr}_{\phi} & \widehat E_{x\times S}\otimes \widehat \omega_{x\times S} \ar[d]_{\wr}^{\phi \otimes dt}\\
\Oc_S[[z]]^n \ar[r]^M & \Oc_S[[z]]^n
}$$
\item Two triples $(E,\varphi,\phi)$ and $(E',\varphi',\phi')$ are said to be equivalent whenever there exists an isomorphism of vector bundles $f\colon E\iso E'$ compatible with all data, that is:
\begin{itemize}
 \item $f$ is compatible with $\phi$ and $\phi'$, so the following diagram is commutative:
$$\xymatrix{
\widehat E_{x\times S}\ar[rr]^{\sim} \ar[dr]_{\phi} & & \widehat E'_{x\times S} \ar[ld]^{\phi'}\\
 & \Oc_S[[z]]^n &
}$$
 \item $f$ is compatible with $\varphi$ and $\varphi'$, that is, the diagram:
$$\xymatrix{
E \ar[r]^{\varphi} \ar[d]^{\wr}_f & E\otimes \omega \ar[d]_{\wr}^{f\otimes \Ide}\\
E' \ar[r]^{\varphi'} & E'\otimes \omega 
}$$
is commutative.
\end{itemize}

\end{enumerate}
\end{defin}

\begin{remark}
Given $(E,\phi,\varphi)\in \Higgs_X^{\infty}(S)$, the elements $(a_1,\dots ,a_n)$ of the third condition are forcely the  characteristic coefficients of $\widehat \varphi$.
\end{remark}

\begin{remark}\label{higgs:r:diagrconm}
 Whenever two $S$-valued points $(E,\varphi,\phi)$ and $(E',\varphi',\phi')$ are equivalent, the characteristic coefficients of both $\widehat \varphi$ and $\widehat \varphi '$ are the same. This implies the commutativity of the following diagram:
$$\xymatrix{
E \ar[rrr]^{\varphi} \ar[dd]^{\wr} \ar[rd] & & & E\otimes \omega \ar[dl] \ar[dd]^{\wr} \\
& \Oc_S[[z]]^n \ar[r]^M & \Oc_S[[z]]^n & \\
E' \ar[ur] \ar[rrr]^{\varphi'} & & & E'\otimes \omega \ar[ul]
}$$
\end{remark}

%\begin{remark}\label{higgs:r:diagrconm}
% Cuando dos ternas $(E,\varphi,\phi)$, $(E',\varphi',\phi')$ son equivalentes, los coeficientes de los polinomios %caracter\'{\i}sticos de los campos de Higgs formales son los mismos y se tiene un diagrama conmutativo:
%$$\xymatrix{
%E \ar[rrr]^{\varphi} \ar[dd]^{\wr} \ar[rd] & & & E\otimes \omega \ar[dl] \ar[dd]^{\wr} \\
%& k[[z]]^n \ar[r]^M & k[[z]]^n & \\
%E' \ar[ur] \ar[rrr]^{\varphi'} & & & E'\otimes \omega \ar[ul]
%}$$
%\end{remark}

\begin{defin}\label{higgs:d:Hmap}
The \emph{formal Hitchin map} is the morphism:
$$\Hc_{\infty} \colon \Higgs_X^{\infty} \to \As$$
which associates to each triple $(E,\phi,\varphi)\in \Higgs_X^{\infty}(S)$ the characteristic coefficients
$$(a_1,\dots , a_n)\in \As(S)\qquad \qquad a_i=\Tr (\Lambda^i \widehat \varphi)$$
of the formal Higgs field $\widehat \varphi$.
\end{defin}
From now on, we will restrict to the open dense subset of $\As$ consisting of those polynomials $p(T)\in \As(S)$ such that $\Oc_S[[z]][T]/p(T)$ is a separable $\Oc_S[[z]]$-algebra.

\begin{thm}\label{higgs:t:repHiggs}
 The functor $\Higgs_X^{\infty}$ is representable by a $k$-scheme.
\end{thm}
\begin{proof}
Let $\U_X^{\infty}$ be the fine moduli space of rank $n$ vector bundles on $X$ with formal trivialization (see \cite{AA}). Then, there exists a universal object $(\Ec,\Phi)$, where $\Ec$ is a rank $n$ vector bundle on $X\times \U_X^{\infty}$ and $\Phi$ a formal trivialization of $\Ec$ along $\{x\}\times \U_X^{\infty}$.

Let $\Vs$ be the fiber bundle associated to the vector bundle on $X\times \U_X^{\infty}$:
$$\Homc_{X\times \U_X^{\infty}} (\Ec,\Ec \otimes \omega)\,,$$
$\omega$ being the pullback of $\w$ to $X\times \U_X^{\infty}$. Let $f\colon \Vs \to \U_X^{\infty}$ be the composition of the natural morphism $\Vs \to X\times \U_X^{\infty}$ with the projection onto $\U_X^{\infty}$. 

Let $\Ec_{\Vs}$ be the vector bundle on $X\times \Vs$ defined by the pullback $(1\times f)^{\ast}\Ec$, denote by
$$\Phi_{\Vs}\colon \widehat \Ec_{\Vs,x\times \Vs}\iso \Oc_{\Vs}[[z]]^n$$
the induced formal trivialization of $\Ec_{\Vs}$ along $x\times \Vs$
and let:
$$\varphi \colon \Ec_{\Vs} \to \Ec_{\Vs}\otimes \omega_{\Vs}$$
be the universal morphism. 

One has that the points of $\Vs$ with values in a $k$-scheme $S$, $g\colon S\to \Vs$, are triples $(E_S,\varphi_S,\phi_S)$ where $E_S$ is the pullback of $\Ec_{\Vs}$ to $X\times S$ by the morphism $1\times g$, $\varphi_S$ is the pullback of $\varphi$ and $\phi_S$ is the pullback of the formal trivialization.

Bearing in mind the induced morphism:
$$\widehat \varphi \colon \widehat \Ec_{\Vs,x\times \Vs} \to \widehat \Ec_{\Vs,x\times \Vs}\otimes \widehat \omega_{\Vs,x\times \Vs}$$
it follows (see Definition \ref{higgs:d:higgs}) that the points of $\Higgs_X^{\infty}$ with values in $S$ are  points of $\Vs$ with values in $S$, $(E_S,\varphi_S,\phi_S)$, for which the composition $(\Phi_{\Vs,S}\otimes dt_{\Vs,S})\circ \widehat \varphi_S \circ \Phi_{\Vs,S}^{-1}$ is defined by a matrix of the type:
$$\left( \begin{array}{cccccc}
0 & & & & 0 & (-1)^{n+1}a_n\\
1 & \cdot & & & \cdot & (-1)^na_{n-1} \\
\cdot & 1 &\cdot  & & \cdot & \cdot \\
\cdot & & \cdot & \cdot & \cdot & \cdot \\
\cdot & & & \cdot & 0 & -a_2 \\
0 & \cdot & & 0 & 1 & a_1
\end{array} \right)$$
for arbitrary $(a_1,\dots ,a_n)\in \As(S)$. One can check that this is a closed condition.
\end{proof}

\subsection{Krichever map.}\qquad
Let us denote $V=k((z))^n$ and $V^+=k[[z]]^n$. Recall that the infinite Grassmannian $\Gr(V)$ associated to the couple $(V,V^+)$ is the infinte dimensional $k$-scheme whose rational points are $k$-vector subspaces $W$ of $V$ such that $W\cap V^+$ and $V/W+V^+$ are finite dimensional subspaces (see \cite{AMP}).

Recall also (\cite{AA}, \cite{Mu}) that the Krichever map
$$\U_X^{\infty} \to \Gr(V)$$
is defined by sending $(E,\phi)$ to $(t \circ \phi) \varinjlim_m \pi_{\ast}E(m)$ and for rational points maps  $(E,\phi)$ to $(t \circ \phi)\Hcoh^0(X-x,E)$.

\begin{defin}\label{higgs:d:Kr}
Let $(E,\phi,\varphi)$ be a point of $\Higgs^{\infty}(S)$. The Krichever map for the functor $\Higgs_X^{\infty}$ is defined by:
\begin{align*}
\Higgs^{\infty}_X(S)& \to \Gr(k((z))^n)(S)\times \As(S)\\
(E,\phi,\varphi) & \mapsto \big ((t \circ \phi) \varinjlim_m \pi_{\ast}E(m),p_{\widehat \varphi}(T)\big )
\end{align*}
where $\pi \colon X\times S \to S$ and $p_{\widehat \varphi}(T)$ is the characteristic polynomial of the formal Higgs field $\widehat \varphi$.
\end{defin}

\begin{remark}
For rational points:
\begin{align*}
\Higgs^{\infty}_X(k)& \to \Gr(k((z))^n)(k)\times \As(k)\\
(E,\phi,\varphi) & \mapsto \big ((t \circ \phi) \Hcoh^0(X-x,E),p_{\widehat \varphi}(T)\big )
\end{align*}
\end{remark}

\begin{thm}\label{higgs:t:KrIny}
 The Krichever map is injective.
\end{thm}

\begin{proof}
 Let $(E,\phi,\varphi)$ and $(E',\phi',\varphi')$ be two points in $\Higgs^{\infty}_X$ with values in the $k$-scheme $S$ and assume that:
$$(t \circ \phi) \varinjlim_m \pi_{\ast}E(m)=(t \circ \phi') \varinjlim_m \pi_{\ast}E'(m)$$
$$p_{\widehat \varphi}(T)=p_{\widehat \varphi'}(T)\,.$$
Since Krichever map is injective for $\U_X^{\infty}$ (see \cite{AA},\cite{Mu}), there exists an isomorphism $E\iso E'$ compatible with formal trivializations $\phi$ and $\phi'$. Because both families verifies the third condition of Definition \ref{higgs:d:higgs} it follows that the diagram:
$$\xymatrix{
E \ar[rrr]^{\varphi} \ar[dd]^{\wr} \ar[rd] & & & E\otimes \omega \ar[dl] \ar[dd]^{\wr} \\
& \Oc_S[[z]]^n \ar[r]^M & \Oc_S[[z]]^n & \\
E' \ar[ur] \ar[rrr]^{\varphi'} & & & E'\otimes \omega \ar[ul]
}$$
is commutative, therefore $(E,\phi,\varphi) \sim (E',\phi',\varphi')$ (see remark \ref{higgs:r:diagrconm}).
\end{proof}

Now we want to characterize the image of the Krichever map. It is known (\cite{AA,Mu}) that the image of the Krichever map:
$$\U_X^{\infty} \hookrightarrow \Gr(V)$$
consists of those points $W\in \Gr(V)$ such that $A\cdot W\subseteq W$ (where $A=\Hcoh^0(X-x,\Oc_X)$, see \cite{MP1}).
Since
$$\Higgs_X^{\infty} \hookrightarrow \Gr(V)\times \As$$
takes values in $\U_X^{\infty}\times \As$, we are interested in studying the image of 
$$\Higgs_X^{\infty} \hookrightarrow \U_X^{\infty}\times \As\,.$$

Consider the natural map:
$$k((z))\otimes_k V \to V=k((z))\otimes_{k((z))}V\,.$$
If $U$ and $W$ are vector subspaces of $k((z))$ and $V$ respectively, we denote $U\cdot W$ the image of $U\otimes_k W \to V$.

\begin{prop}\label{higgs:p:WU}
 Take $L$ be a line bundle on $X$ and $\phi_L$ a formal trivialization of $L$ at $p$. Consider $(E,\phi_E)\in \U_X^{\infty}(k)$ and let:
$$W=\Hcoh^0(X-x,E)\in \Gr(V)(k)\,,\qquad U=\Hcoh^0(X-x,L)\in \Gr(k((z)))(k)$$
be the points they define via the Krichever map. Then:
$$U\cdot W=\Hcoh^0(X-x,L\otimes_{\Oc_X}E)\in \Gr(V)(k)\,.$$
\end{prop}
\begin{proof}
It follows from the surjectivity of
$$\Hcoh^0(X-x,L)\otimes_k \Hcoh^0(X-x,E) \to \Hcoh^0(X-x,L\otimes_{\Oc_X}E)\,.$$
Recall that $X-x=\Spec A$ is an affine open subset of $X$.
\end{proof}

 One can easily generalizes the above result for points with values in any $k$-scheme $S$. This implies that we can define the following morphism of schemes:
\begin{align*}
 h_U \colon \U_X^{\infty} & \to \U_X^{\infty}\\
W & \mapsto U\cdot W
\end{align*}

\begin{remark}
Note that in general we don't have a well-defined scheme morphism:
\begin{align*}
 \Gr(k((z)))\times \Gr(V) & \to \Gr(V)\\
(U,W) &\mapsto U\cdot W
\end{align*}
\end{remark}

Let $(a_1,\dots ,a_n)$ be a point in $\As$ with values in $S$ and let $T$ be the endomorphism given by the matrix:
\begin{equation}\label{higgs:e:T}
 T=\left( \begin{array}{cccccc}
0 & & & & 0 & (-1)^{n+1}a_n\\
1 & \cdot & & & \cdot & (-1)^na_{n-1} \\
\cdot & 1 &\cdot  & & \cdot & \cdot \\
\cdot & & \cdot & \cdot & \cdot & \cdot \\
\cdot & & & \cdot & 0 & -a_2 \\
0 & \cdot & & 0 & 1 & a_1
\end{array} \right)\in \End_{\Oc_S((z))}\Oc_S((z))^n
\end{equation}

\begin{defin} \label{d:est}
Let $W\in Gr(V)(S)$ be a point with values in $S$. A sub-$\Oc_S$-algebra $A$ of $\Oc_S((z))$ is said to \emph{stabilize} $W$ if it verifies:
\begin{enumerate}
\item $\Oc_S((z))/A$ is flat over $S$.
\item $A\cdot W\subseteq W$.
\end{enumerate} 
The \emph{stabilizer algebra} of $W$ is the maximum subalgebra of $\Oc_S((z))$, $A_W$, which stabilizes $W$.
\end{defin} 
%It exists because of Zorn lemma, since the family of sub-$\Oc_S$-algebras which stalize $W$ is not empty (it %contains $\Oc_S$) and the union of subalgebras stabilizing $W$ also stabilize $W$. 

If $S=\Spec k$ is the spectrum of a field $k$, flatness is straighforward and the stabilizer algebra coincides with those of \cite{Mu}:
$$ A_W=\{f\in k((z)) \vert f\cdot W\subseteq W\}\,.$$

\begin{prop}\label{higgs:p:T(W)}
Let $S$ be a $k$-scheme and $A_S:=\Hcoh^0(X-x,\Oc_X)\otimes_k S$. Given $W\in \U_X^{\infty}(S)$ it holds:
\begin{enumerate}
 \item $T(W)\in \Gr(V)(S)$.
 \item $T(W)$ is $A_S$-module.
 \item $A_{T(W)}=A_W=A_S\in \Gr(k((z)))(S)$.
\end{enumerate}
In particular, $T(W)\in \U_X^{\infty}(S)$.
\end{prop}
\begin{proof}
It follows from the $\Oc_S((z))$-linearity of $T$ and the properties of the stabilizer algebra $A_W$ of $W$.
\end{proof}

We can now define another endomorphism of $\U_X^{\infty}$:
\begin{align*}
 T \colon \U_X^{\infty} & \to \U_X^{\infty}\\
W & \mapsto T(W)
\end{align*}

\begin{thm}\label{higgs:t:carHiggs}
Let $\Omega$ be the point in $\Gr(k((z)))(S)$ defined by the relative canonical sheaf $\omega_{X\times S/S}$.

A couple $(W,p(T))\in (\U_X^{\infty}\times \As)(S)$ lies on the image of the morphism
$$\Higgs_X^{\infty}\hookrightarrow \U_X^{\infty}\times \As$$ 
if and only if the restriction of $T$ (equation \ref{higgs:e:T}) to $W$ takes values into $W\cdot \Omega$:
$$\xymatrix{
\Oc_S((z))^n \ar[r]^T & \Oc_S((z))^n \\
W \ar@{^(->}[u] \ar[r]^{T_{|W}} & W\cdot \Omega \ar@{^(->}[u]
}$$
For rational points this condition translates into $T(W) \subseteq W\cdot \Omega$.

%Analogamente, la imagen del morfismo de Krichever coincide con el n\'ucleo de las aplicaciones:
%$$\xymatrix{
%\U_X^{\infty}\ar@<4pt>[r]^{T} \ar@<-2pt>[r]_{h_{\Omega}} &\U_X^{\infty}
%}$$
\end{thm}
\begin{proof}
Direct implication follows from the definition of the Krichever morphism. Let's see how to recover the geometric data from $(W,p(T))\in \U_X^{\infty}(S)\times \As(S)$. By hypothesis, $T_{|W}$ takes values in $W\cdot \Omega$ and we already know that $W\in \U_X^{\infty}(S)$, so we can recover a rank $n$ vector bundle on $X\times S$ together with its formal trivialization $\phi$ (see \cite{AA}, \cite{Mu}). By Proposition \ref{higgs:p:WU} we have that $ W\cdot \Omega$ is the point in $\Gr(k((z))^n)(S)$ defined by $E\otimes_{\Oc_X}\omega_{X\times S/S}$, that is:
$$W\cdot \Omega=\varinjlim_m \pi_{S,\ast} (E\otimes_{\Oc_X}\omega_{X\times S/S})(m(x\times S))\,.$$
Therefore, $T_{|W}$ translate into the morphism:
$$T_{|W}\colon \varinjlim_m \pi_{S,\ast} E(m(x\times S)) \to \varinjlim_m \pi_{S,\ast} (E\otimes_{\Oc_X}\omega_{X\times S/S})(m(x\times S))$$
where $\pi_S \colon X\times S \to S$. Thus, to recover the Higgs field 
$$E\to E\otimes_{\Oc_X}\omega_{X\times S/S}$$
we have to extend $T_{|W}$ to $x\times S$. But, this follows from the fact that $T$ is $\Oc_S((z))$-linear (see equation \ref{higgs:e:T}). In addition, the compatibility condition between $\varphi$ and $\phi$ is obtained by construction.
%(luego conmuta con los estabilizadores) y por la proposici\'on \ref{higgs:p:T(W)} sabemos que los estabilizadores de $W$ y %$T(W)\subseteq W\cdot \Omega$ son el mismo, podemos extender el morfismo a un morfismo de $\Oc_{X\times S}$-m\'odulos:
%$$\varphi \colon E\to E\otimes_{\Oc_{X\times S}}\omega_{X\times S}$$
%(en otras palabras, $T$ es compatible con las filtraciones dadas por el orden del polo en el punto $p$). Por construcci\'on %$\varphi$ es compatible con la trivializaci\'on formal $\phi$ en el sentido de la definici\'on de $\Higgs_X^{\infty}$.
\end{proof}

This result allow us to prove the following theorem:

\begin{thm}\label{higgs:t:rep_en_U_X}
 The functor $\Higgs_X^{\infty}$ is representable by a closed subscheme of $\U_X^{\infty}\times \As$.
\end{thm}

\begin{proof}
By \ref{higgs:t:carHiggs}, the image of:
$$\Higgs_X^{\infty} \hookrightarrow \U_X^{\infty} \times \As$$
consists of pairs $(W,p(T))$ such that $T_{|W}(W)\subseteq W\cdot \Omega$. By proposition \ref{higgs:p:WU} we have that $W\cdot \Omega \in \Gr(V)$. The statement then follows from the next result  (see \cite[Lemma 2.4.6]{AA}):

Let $S$ be a $k$-scheme and let $\F$ be a sheaf of $\Oc_S$-modules. Let $\F_1$ and $\F_2$ be two quasi-coherent  subsheaves of $\F$ such that locally $\F /\F_2 \iso \varprojlim_i L_i$, where $L_i$ are free coherentsheaves.
Then, the points $f\colon S'\to S$ such that $f^*(\F_1)\subseteq f^*(\F_2)$, as subsheaves of $f^*(\F)$, are the points of a closed subset of $S$.
\end{proof}

\section{Spectral Construction.}\label{higgs:s:spectral}\quad

We call formal base curve to:
$$\widehat X_x:=\Spf k[[z]]\,.$$
%Recall that we also have a formal trivialization of $\omega_X$:
%$$\widehat \omega_{X,x}\iso \widehat \Oc_{X,x}\iso k[[z]]$$
%and write
%$$\Ts^{\ast}\widehat X_x:=\Spf \Sc^{\bullet}\widehat \omega_{X,x}^{-1}\iso \Spf k[[z]][T]$$
%for the formal cotangent space to $X$ at $p$, where $\Sc^{\bullet}$ denotes the symmetric algebra over $\widehat %\Oc_{X,x}$ and $T$ is an indeterminate.
%
%Given a rational point $p(T)\in \As(k)$, that is, a monic polynomial in $T$ with coefficients in $\widehat %\omega_{X,x}\iso k[[z]]$:
%$$p(T)=T^n +a_1 T^{n-1}+ \cdots +a_n$$
%(see \ref{higgs:p:As}), we define $\widehat{\Ic}$ as the sheaf of ideals of $\Sc^{\bullet}\widehat %\omega_{X,x}^{-1}$ generated by the image of:
% \begin{align*}
%\widehat \omega_{X,x}^{-n} & \to \oplus_{i\geq 0} \widehat \omega_{X,x}^{-i} \\
% \xi & \mapsto (\xi \otimes a_n,\xi \otimes a_{n-1},\dots ,\xi \otimes a_1,\xi, 0, 0, \dots )
%\end{align*}

 \begin{defin}
Let $p(T)$ be a point in $\As(S)$. We define the \emph{formal spectral curve} associated to $(\widehat X,p(T))$ as:
$$\Spf \Oc_S[[z]][T]/p(T)\,.$$
\end{defin}

Let $(E,\phi,\varphi)$ be  point in $\Higgs_X^{\infty}$ with values in $S$. Recall that the formal Hitchin map (see Definition \ref{higgs:d:Hmap}):
$$\Hc^{\infty} \colon \Higgs_X^{\infty} \to \As$$
sends $(E,\phi,\varphi)$ to the characteristic polynomial $p_{\widehat \varphi}(T)\in \As(S)$ of the formal Higgs field  $\widehat \varphi$.

\begin{prop}\label{higgs:p:Higgs_espectral}
Let $S$ be a $k$-scheme. Given $(E,\phi,\varphi)\in \Higgs_X^{\infty}(S)$, there exists a formal trivialization of $E$ along $x\times S$:
$$\phi_{\varphi} \colon \widehat E_{x\times S}\iso \Oc_S[[z]][T]/p_{\widehat \varphi}(T)$$
in such a way that  $M$ (see Definition \ref{higgs:d:higgs}) corresponds to the multiplication by $T$ in $\Oc_S[[z]][T]/p_{\widehat \varphi}(T)$.
\end{prop}

\begin{proof}
The proof uses the classical correspondence between a Higgs pair and the Jacobian of the spectral curve of \cite{BNR}. The formal Higgs field $\widehat \varphi$ endows $\widehat E_{x\times S}$ with a line bundle structure over the formal spectral curve $\Spf \Oc_S[[z]][T]/p_{\widehat \varphi}(T)$, so that we can choose a trivialization:
$$\widehat E_{x\times S}\iso \Oc_S[[z]][T]/p_{\widehat \varphi}(T) \,.$$
Using the given formal trivialization:
$$\phi \colon \widehat E_{x\times S}\iso \Oc_S[[z]]^n$$
we can define an $\Oc_S[[z]]$-algebra structure on $\Oc_S[[z]]^n$ and, because of the compatibility condition between $\phi$ and $\varphi$ of Definition \ref{higgs:d:higgs}, we get that $M$ corresponds to the multiplication by $T$ in $\Oc_S[[z]][T]/p_{\widehat \varphi}(T)$.
\end{proof}

%\begin{defin}\label{higgs:d:higgs}
% Se define el functor $\Higgs_X^{\infty}$ como el hacificado de:
%\begin{align*}
%\mathcal C & _ {k-esq} \rightsquigarrow \mathcal C_{sets}\\
%& S \mapsto \{(E,\phi,\varphi)\}/\sim
%\end{align*}
%donde
%\begin{enumerate}
%\item $E$ es un fibrado de rango $n$ sobre $X\times S$.
%\item $\varphi \colon E \to E \otimes \omega_{X\times S/S}$ es un campo de Higgs relativo.
%\item $\phi$ es una trivializaci\'on formal de $E$ a lo largo de $x\times S$, es decir, entendida como un %isomorfismo de $\widehat{(\Sc^{\bullet}\omega^{-1}/\Ic)_{x\times S}}$-m\'odulos:
% $$\phi \colon \widehat{E_{x\times S}}\iso \widehat{(\Sc^{\bullet}\omega^{-1}/\Ic)_{x\times S}}$$
%\item Diremos que $(E,\varphi,\phi)\sim(E',\varphi',\phi')$ cuando exista un isomorfismo de fibrados $E\iso E'$ %compatible con todos los datos. 
%\end{enumerate}
%\end{defin}

\subsection{Relative Infinite Grassmannian.}\label{higgs:ss:Grc}\quad

Observe that for each $p(T)\in \As(S)$ the $\Oc_S((z))$-module $V_S=\Oc_S((z))^n$ can be enriched with a natural  $\Oc_S((z))$-algebra structure:
$$V_S \iso \Oc_S((z))[T]/p(T)$$ 
using the standard basis $\{1,T,T^2,\dots ,T^{n-1}\}$. 

Now take a rational point $(E,\phi,\varphi)\in \Higgs_X^{\infty}(k)$ and denote by $\pi \colon X_{\varphi} \to X$ the spectral cover.

%$\pi_{\varphi} \colon X_{\varphi}\to X\times S$  the spectral cover, we have:
%$$\widehat{\pi_{\varphi,\ast}\Oc_{X_{\varphi},p}}\iso \Oc_S[[z]][T]/p_{\widehat \varphi}(T)$$
Using the formal trivialization of the Proposition \ref{higgs:p:Higgs_espectral}:
$$\phi_{\varphi} \colon \widehat E_x\iso k[[z]][T]/p_{\widehat \varphi}(T)$$
is easy to check that  
$$W=\phi_{\varphi}\Hcoh^0(X_{\varphi}-\pi^{-1}(x),E)$$
is a rational point of the infinite Grassmannian $\Gr(k((z))[T]/p_{\widehat \varphi}(T))$. 

The aim for this section is to make this construction uniform as $p(T)$ varies. This will lead us to define the relative infite Grassmannian. The construction is similar to that of the standard infinite Grassmannian (\cite{AMP}), so we will omit the proofs (see \cite{P} for further details).

%Note that these discrete sub-$\Oc_S$-module are not points with values in $S$ of the infinite Grassmannian
%$$\Gr(k((z))[T]/``p(T)")\,,$$
%because for each triple $(E,\phi,\varphi)$ the polynomial may change. That is, we have to find out an infinite %Grassmannian that allows us to make $p(T)$ vary, that allows to vary the spectral curve. Now we know that varying %$p(T)$ means moving over $\As$, so que question is: there exists a relative infinite Grassmannian over $\As$ whose %fibers over $S$-valued points are the Grassmannians of the above discrete sub-$\Oc_S$-modules?

%The answer is yes, and the relative infinte Grassmanniana we are about to define is in fact trivial over the base, %that is, it is isomorphic to $\Gr(k((z))^n)\times S$.

Let $p^{univ}(T)$ denote the universal polynomial of $\As$, that is, the point corresponding to the identity morphism $\Ide \colon \As \to \As$ (recall that by proposition \ref{higgs:p:As} $\As$ is representable) and think of $\Oc_{\As}[[z]][T]/p^{univ}(T)$ as the ring of the universal formal spectral cover.

Denote
$$V=\Oc_{\As} ((z))[T]/p^{univ}(T) \qquad V^+=\Oc_{\As} [[z]][T]/p^{univ}(T)$$
and consider the triple $(V,V^+,V_1)$ where $V_1=zV^+$. We have:

\begin{itemize}
\item $V^+$ is a sheaf of $V_1$-adics $\Oc_{\As}$-algebras ($V_1$ is a sheaf of ideals of $V^+$).
\item $V^+/V_1\iso \Oc_{\As}$ as $\Oc_{\As}$-algebras, $V_1$ is locally principal and $V$ is the localization of $V^+$ by the generator of $V_1$. Equivalentely, for each point of $\As$ there exists an open neighborhood $U$ such that:
$$(V_U,V^+_U,(V_1)_U)\iso (\Oc_U((z))[T]/p^U(T),\Oc_U[[z]][T]/p^U(T),z\Oc_U[[z]][T]/p^U(T))$$
\end{itemize}

Define the following sub-$V^+$-modules of $V$ by $V_i:=V_1^i$, that is, 
$$V_i=z^i\Oc_U[[z]][T]/p^U(T))$$
locally in $\As$.

There also exists a notion of conmensurability for sub-$\Oc_{\As}$-modules of $V$ and verifies the same properties as the standard conmensurability (see \cite{AMP}).

\begin{defin}
The relative infinite Grassmannian associated to the triple $(V,V^+,V_1)$ is the functor that associates to each $\As$-esquema $S$ the following set:
{\small$$\Grc(V)(S)=\left\{ \begin{gathered}
\text{quasicoherents sub-$\Oc_S$-modules  $W\subseteq V_S$ such that $V_S/W$ is}\\
\text{flat and for each point in $S$ there exists an open neighborhood $U$ and }\\
\text{$C$ conmensurable with $V^+$ such that $V_U/(W_U+C_U)=0$}\\
\text{and $W_U \cap C_U$ is loc. free of finite type}
\end{gathered}
\right\}$$}
\end{defin}

\begin{thm}
$\Grc(V)$ is representable by a $\As$-scheme (thus, in particular by a $k$-scheme) that we will still denote  $\Grc(V)$.
\end{thm}

\begin{thm}
For all $\As$-scheme $S$ denote $V_S=V\widehat \otimes_{\Oc_{\As}}\Oc_S$. It holds:
$$\Grc(V_S)\iso \Grc(V)\times_{\As}S$$
In particular, for each rational point $p(T) \colon \Spec k \to \As$:
$$\Grc(V_k)=\Grc(k((z))[T]/p(T))$$
\end{thm}
This means that: 
$$\Grc(\Oc_{\As}((z))[T]/p^{univ}(T))\to \As$$
is a fibration of infinite Grassmannians.

%Consider the following diagram:
%$$\xymatrix{
%\Grc(V)\times_{\As,f} S \ar[r] \ar[d] & \Grc(V) \ar[d]\\
%S \ar@/^/@{-->}[u]  \ar@{-->}[ur] \ar[r]^{f} & \As
%}$$
%We have:
%$$\Grc(V)(S)=\displaystyle{\coprod_{f\in \As(S)}}\Gamma(\Grc(V)\times_{\As,f}S / S)$$

\begin{remark}\label{higgs:r:Grc&Gr}
Note that one has an isomorphism of $\Oc_{\As}$-modules
$$\Oc_{\As}((z))[T]/p^{univ}(T)\iso \Oc_{\As}((z))^n\,,$$
therefore, considering the trivial fibration of $k$-schemes
$$\Gr (k((z))^n)\times \As \to \As$$
one can get an isomorphism:
\begin{equation}\label{higgs:e:Grc&Gr}
\Gr (k((z))^n)\times \As\iso \Grc(\Oc_{\As}((z))^n)\iso \Grc(\Oc_{\As}((z))[T]/p^{univ}(T))
\end{equation}
\end{remark}

\subsection{Krichever map via the spectral construction.}\quad

Given a point $(E,\phi,\varphi)\in \Higgs_X^{\infty}(S)$ and using the formal trivialization $\phi_{\varphi}$ of  Proposition \ref{higgs:p:Higgs_espectral}, one gets a well-defined Krichever map:
\begin{align*}
 \Higgs_X^{\infty}(S) & \to \Grc(\Oc_{\As}((z))[T]/p^{univ}(T))(S)\\
(E,\phi,\varphi) & \mapsto \phi_{\varphi}(\varinjlim_m \pi'_{\ast}E_{\varphi}(m))
\end{align*}
where $\pi' \colon X_{\varphi}\to S$. For rational points is given by:
\begin{align*}
\Higgs_X^{\infty}(k) & \to \Grc(\Oc_{\As}((z))[T]/p^{univ}(T))(k)\\
(E,\phi,\varphi) & \mapsto \phi_{\varphi}\Hcoh^0(X_{\varphi}-\pi^{-1}(x),E_{\varphi})
\end{align*}
$\pi\colon X_{\varphi} \to X$ being the spectral cover.

Bearing in mind the isomorphism of the equation (\ref{higgs:e:Grc&Gr}), 
%$$\Gr (k((z))^n)\times \As\iso \Grc(\Oc_{\As}((z))[T]/p^{univ}(T))$$
it follows that this Krichever map is in fact equivalent to the Krichever map defined in \ref{higgs:d:Kr}, and by theorem  \ref{higgs:t:KrIny} is injective. The image is characterized in a similar way to that of theorem \ref{higgs:t:carHiggs}.

See also \cite{DM} for other interpretation of the spectral construction.

\section{Equations of the moduli space of Higgs Pairs.}\qquad

Let $p(T)\in \As(k)$ be a rational point and denote
$$V_p=k((z))[T]/p(T) \qquad \quad V_p^+=k[[z]][T]/p(T)\,.$$
%We call formal base curve, associated to the couple $(V_p,V_p^+)$, to
%$$\widehat X:=\Spf k[[z]]$$
%and formal spectral curve, associated to $(V_p,V_p^+)$, to:
%$$\widehat X_V:=\Spf V_p^+\,.$$
From now on, assume that  $V_p$ is a separable $k((z))$-algebra with the following decomposition:
$$V_p\iso V_1 \times \cdots \times V_r$$
where $V_i=k((z))[T_i]/T_i^{n_i}-zu_i(T_i)$, $n_1+ \cdots + n_r=n$ and $u_i(T_i)$ are invertibles in $k[[T_i]]$.
In addition, it is easy to prove the following isomorphism:
\begin{equation}\label{eq:e:descom}
 k((T_i))\iso V_i=k((z))[T_i]/T_i^n-zu(T_i)
\end{equation}
where $z$ is thought as a serie in $T_i$ by the expresion $z=T_i^{n_i}u_i(T_i)^{-1}$.

This assumption is motivated by the following geometric fact:
\begin{prop}\label{higgs:p:X_V}
Consider $(E,\varphi,\phi)\in \Higgs_X^{\infty}(k)$ and let $p_{\widehat \varphi}(T)\in \As(k)$ be characteristic polynomial of the induced formal Higgs field. Then, the ring of the formal spectral curve, $\widehat X_V:=\Spf V_p^+$, is isomorphic, as $k[[z]]$-algebra, to:
$$k[[z]][T_1]/T_1^{n_1}-zu_1(T_1) \times \cdots \times k[[z]][T_r]/T_r^{n_r}-zu_r(T_r)\,,$$
where $n_1+\cdots +n_r=n$ and $u_i(T_i)\in k[[T_i]]^{\ast}$.
\end{prop}

\begin{proof}
Indeed, denote $\widehat \Oc_x=k[[z]]$ and $\widehat \Oc_y=V_p^+=k[[z]][T]/p_{\widehat \varphi}(T)$. $\widehat \Oc_x$ is a local, regular and complete ring with maximal ideal $\mathfrak m_x=(z)$ and $\widehat \Oc_y$ is a regular, complete, separable $\widehat \Oc_x$-algebra of dimension one, thus, $\widehat \Oc_y$ is a Dedekind domain and we have a primary decomposition:
$$\mathfrak m_x \Oc_y=\mathfrak m_1^{n_1}\cdots \mathfrak m_r^{n_r}$$
where $\mathfrak m_i$ are the maximal ideals of the points:
$$\{y_1,\cdots ,y_r\}=\Spec \Oc_y\,.$$
Moreover, if we denote $\widehat \Oc_{y_i}$ the completion of $(\Oc_y)_{y_i}$, in $\widehat \Oc_{y_i}$ one has $\mathfrak m_x\Oc_{y_i}=\mathfrak m_i^{n_i}$. Since $\widehat \Oc_{y_i}$ are local, regular and complete $k$-algebras of dimension one, by Cohen theorem it follows that $\widehat \Oc_{y_i}\iso k[[T_i]]$ and then, in $\widehat \Oc_{y_i}$, it is:
$$z=T_i^{n_i}v_i(T_i)$$
$v_i(T_i)$ being an invertible element in $\widehat \Oc_{y_i}$. 
In addition, we know that $k[[z]] \to k[[T_i]]$ is an integer morphism, so one gets a relation:
$$T_i^{n_i}+b_1^{(i)}T_i^{n_i-1}+\cdots +b_{n_i}^{(i)}=0$$
where $b_j^{(i)}\in k[[z]]$. That is:
$$k[[T_i]]\iso k[[z]][T_i]/T_i^{n_i}+b_1^{(i)}T_i^{n_i-1}+\cdots +b_{n_i}^{(i)}$$
Moreover, if we restrict to $z=0$ we should get a point counted $n_i$ times, what means that $b_j^{(i)}=z\bar b_j^{(i)}$ where $\bar b_j^{(i)} \in k[[z]]$. Therefore:
$$T_i^{n_i}=-z(\bar b_1^{(i)}T_i^{n_i-1}+\cdots +\bar b_{n_i}^{(i)})\,.$$ 
If we denote $u_i(T_i)$ to $v_i(T_i)^{-1}$, one gets that $u_i(T_i)$ is a polynomial in $T_i$ with coefficients in $k[[z]]$ and:
$$k[[T_i]]\iso k[[z]][T_i]/T_i^{n_i}-zu_i(T_i)$$

Because of separabily and the fact that we are working in characteristic zero, if $u_i=v_j$ the $n_i\neq n_j$. Now we can conclude:
$$k[[z]][T]/p_{\widehat \varphi}(T)\iso k[[z]][T_1]/T_1^{n_1}-zu_1(T_1) \times \cdots \times k[[z]][T_r]/T_r^{n_r}-zu_r(T_r)$$
where $u_i(T_i)=v_i(T_i)^{-1}$.
\end{proof}

\begin{remark}
 Note that, since $k$ has characteristic zero, there exists $\bar v_i \in k[[T_i]]$ such that $v_i=\bar v_i^{n_i}$. Then, $\bar T_i=T_i\bar v_i  \in k[[T_i]]$ is a formal parameter for $y_i$ and $z=\bar T_i^{n_i}$, that is:
$$\widehat \Oc_{y_i}\iso \widehat \Oc_x[\bar T_i]/\bar T_i^{n_i}-z\,.$$
In our case we are not allow to do this change of parameter, because once we take a point in $\Higgs_X^{\infty}$ we get a polynomial in $\As$, and if we were allowed to do this change, then the polynomial would be different.
\end{remark}

In order to give the equations describing $\Higgs_X^{\infty}$, we will firstly need to adapt some results concerning Tau and Baker-Akhiezer functions given in \cite{AMP,MP1,MP4}.

\subsection{Formal Jacobian and Abel morphism.}\qquad

Because of the isomorphisms of the equation \ref{eq:e:descom} and following \cite{AMP, MP4} it can be proven that functor $\underline V_p^*$ of invertible elements in $V_p$ is representable by a formal group $k$-scheme whose connected component of the origin decomposes as:
$$\Gamma_{V_p}=\Gamma_{V_p}^- \times \mathbb G_m^r \times \Gamma_{V_p}^+$$
and 
$$\Gamma_{V_p}^{\pm}\iso \Gamma_{V_1}^{\pm} \times \cdots \times \Gamma_{V_r}^{\pm}\,.$$

Moreover, the formal Jacobian of the formal spectral curve $\widehat X_V$, $\Jc(\widehat X_V)$, is isomorphic to $\Gamma_{V_p}^-$ (see \cite[Theorem 4.14]{AMP} and also \cite{MP4}).

Let $\widehat \As_{\infty}$ be the formal group scheme:
$$\widehat \As_{\infty}:=\varinjlim_n \Spec k[[t_1,\dots ,t_n]]$$
endowed with the additive group law. Let us denote $k\{\{t_1,t_2,\dots \}\}$ for the ring $\Oc_{\widehat \As_{\infty}}=\varprojlim_n k[[t_1,\dots ,t_n]]$ and set:
$$\As_{\infty}^r:=\widehat \As_{\infty} \times \overset{r}{\dots} \times \widehat \As_{\infty}\,.$$

Since the characteristic of $k$ is zero, the exponetial map is the isomorphism of formal group $k$-schemes (see \cite[Definition 4.11]{AMP}):
\begin{align*}
\expo \colon \widehat \As_{\infty}^r & \to  \Jc(\widehat X_V)\\
(\{a_i^{(1)}\}_{i>0},\dots ,\{a_i^{(r)}\}_{i>0}) & \mapsto \big( \expo(\sum_{i>0}\frac{a_i^{(1)}}{T_1^i}), \dots ,\expo(\sum_{i>0}\frac{a_i^{(r)}}{T_r^i})\big)
\end{align*}

Now we can identify $\Jc(\widehat X_V)$ with the formal spectrum of the ring:
$$k\{\{t_1^{(1)}, \dots \}\}\hat \otimes \cdots \hat \otimes k\{\{t_1^{(r)}, \dots \}\}$$
and its universal element will be denoted by:
$$v=\prod_{i=1}^r \expo(\sum_{j\geq 1}\frac{t_j^{(i)}}{T_i^j})\,,$$

The Abel morphism of degree one (see \cite{MP4}):
$$\phi_1 \colon \widehat X_V \to \Jc(\widehat X_V)$$ 
is defined by the $r$ series: 
$$\big( (1-\frac{\bar T_1}{T_1})^{-1},\dots ,(1-\frac{\bar T_r}{T_r})^{-1}\big)$$
where we distinguish the indeterminates:
$$\bar T_i \in \Oc_{\widehat X_V}=k[[\bar T_1]] \times \cdots \times k[[\bar T_r]]\iso V_p^+$$
and $T_i\in \Oc_{\Jc(\widehat X_V)}$.

Equivalentely, the Abel morphism is the induced morphism by the next ring homomorphism:
\begin{align*}
k\{\{t_1^{(1)},t_2^{(1)}, \dots \}\}\hat \otimes \cdots \hat \otimes k\{\{t_1^{(r)},t_2^{(r)}, \dots \}\} &\to k[[\bar T_1]] \times \cdots \times k[[\bar T_r]]\\
t_i^{(j)}& \mapsto \bar T_j^{(i)}
\end{align*}

\subsection{Tau and Baker-Akhiezer functions .}\qquad

Consider the fibration of infinite Grassmannians:
$$\Gr(k((z))^n)\times \As \to \As$$
and let $p(T)$ be a rational point of $\As$. The fiber at $p(T)\in \As(k)$ is $\Gr(V_p)(k)$ (see \ref{higgs:ss:Grc}), where $V_p=k((z))[T]/p(T)$.

Consider now the action of $\Gamma_{V_p}$ in $\Gr(V_p)$ by homoteties:
$$\mu \colon \Gamma_{V_p}\times \Gr(V_p) \to \Gr(V_p)$$
and define the Poincar\'e sheaf on $\Gamma_{V_p}\times \Gr(V_p)$ by:
$$\Pc:=\mu^{\ast}\Det_{V_p}^*\,.$$
\begin{defin}
For each rational point $W\in \Gr(V_p)$, we define the \emph{sheaf of $\tau$ functions} of $W$ as the line bundle on $\Gamma_{V_p} \times \{W\}$ given by:
$$\widetilde \Lca_{\tau}(W):=\Pc_{|\Gamma_{V_p} \times \{W\}}$$
We call $\tau$-section of a point $W$, $\tilde \tau_W$, the image of $\mu^{\ast}\Omega_+$ under the morphism
$$\Hcoh^0(\Gamma_{V_p} \times \Gr(V_p),\Pc) \to \Hcoh^0(\Gamma_{V_p} \times \{W\},\widetilde \Lca_{\tau}(W))$$
where $\Omega_+$ is the canonical global section of $\Det_{V_p}^{\ast}$ defined in \cite{AMP}.
\end{defin}
Note that  $\tilde \tau_W$ is not a true function over $\Gamma_{V_p} \times \{W\}$, because $\widetilde \Lca_{\tau}(W)$ is in general non trivial. The algebraic analogous of the tau  function defined by Sato, Segal and Wilson (\cite{SW}) is obtained by restricting $\widetilde \Lca_{\tau}(W)$ to the formal subgroup $\Gamma_{V_p}^- \iso \Jc(\widehat X_V) \subset \Gamma_{V_p}$. Explicitily:

Define $\Lca_{\tau}(W):=\widetilde \Lca_{\tau}(W)_{|\Gamma_{V_p}^- \times \{W\}}$. Like in \cite{AMP}, it can be shown that this vector bundle is trivial over $\Gamma_{V_p}^- \times \{W\}$. In order to obtain a trivialization (that allows us to identify sections with functions) one needs to find out a section without zeros. Let $\Vs$ be the fiber bundle associated to $\Lca_{\tau}(W)$, fix a non-zero element $\rho_W$ in the fiber of $\Vs$ over the identity point $(1,W)\in \Gamma_{V_p}^- \times \{W\}$ and let $\sigma_0$ be the unique $\Gamma_{V_p}^-$-equivariant section such that $\sigma_0(1,W)=\rho_W$. 

Thus, $\sigma_0$ is a constant section without zeros and the global section defined by $\tilde \tau_W$ is identified, via the trivialization given by $\sigma_0$, with the function:
$$\tau_W \in \Oc_{\Gamma_{V_p}^-}\iso \Oc_{\Jc(\widehat X_V)}\iso k\{\{t_1^{(1)}, \dots \}\}\hat \otimes \cdots \hat \otimes k\{\{t_1^{(r)}, \dots \}\}$$
defined by:
$$\tau_W(g_{\bullet})=\frac{\tilde \tau_W(g_{\bullet})}{\sigma_0(g_{\bullet})}=\frac{\mu^{\ast}\Omega_+(g_{\bullet})}{\sigma_0(g_{\bullet})}=\frac{\Omega_+(g_{\bullet}\cdot W)}{g_{\bullet} \cdot \rho_W}$$
for $g_{\bullet}=\prod_{i=1}^r \expo(\sum_{j\geq 1}\frac{t_j^{(i)}}{T_i^j})\in \Jc(\widehat X_V)\iso \Gamma_{V_p}^- \,.$

\begin{remark}
 Over the formal base curve $\widehat X=\Spf k[[z]]$, if $\Omega \in \Gr(k((z)))$, its $\tau$ function over $\Jc(\widehat X)\iso \Gamma^- \iso \Spf k\{\{t_1, \dots \}\}$ is, in characteristic zero (\cite{AMP}):
$$\tau_{\Omega}(g)=\frac{\tilde \tau_{\Omega}(g)}{\sigma_0(g)}=\frac{\mu^{\ast}\Omega_+(g)}{\sigma_0(g)}=\frac{\Omega_+(gW)}{g \rho_W}$$
for $g=\expo(\sum_{j\geq 1}\frac{t_j}{z^j})\in \Jc(\widehat X)\iso \Gamma^- \,.$
\end{remark}

Now is time to go for Baker-Akhiezer functions.

Consider the composition:
$$\tilde \beta \colon \widehat X_V \times \Gamma_{V_p} \times \Gr(V_p) \stackrel{\phi_1 \times \Ide}{\rightarrow} \Gamma_{V_p} \times \Gamma_{V_p} \times \Gr(V_p) \stackrel{m\times \Ide}{\rightarrow} \Gamma_{V_p}\times \Gr(V_p)$$
where $\phi_1 \colon \widehat X_V \to \Gamma_{V_p}$ is the Abel morphism of degree one (with values in $\Gamma_{V_p}^- \iso \Jc(\widehat X_V) \subset \Gamma_{V_p}$) and $m$ is the group law in $\Gamma_{V_p}$.

\begin{defin}
The sheaf of Baker-Akhiezer functions (BA) is the line bundle over $\widehat X_V \times \Gamma_{V_p} \times \Gr(V_p)$ defined by:
$$\widetilde \Lca_{BA}:=\tilde \beta^{\ast} \Pc \,.$$
The sheaf of BA functions of a point $W\in \Gr(V_p)$ is:
$$\widetilde \Lca_{BA}(W):=\widetilde{{\Lca}_{BA}}_{|\widehat X_V \times \Gamma_{V_p} \times \{W\}}\,.$$
Similarly:
$$\widetilde \Lca_{BA}(W):=\tilde \beta_W^{\ast} \tilde \Lca_{\tau}(W)$$
$\tilde \beta_W$ being the restriction of $\tilde \beta$ to $W\in \Gr(V_p)$.
\end{defin}

\begin{defin}
 The BA-section of $W\in \Gr(V_p)$ is $\tilde \psi_W:=v^{-1}\cdot \beta_W^{\ast}(\tilde \tau_W)$, where:
$$\beta_W^{\ast} \colon \Hcoh^0(\Gamma_{V_p}\times \{W\},\widetilde \Lca_{\tau}(W))\to \Hcoh^0(\widetilde X_V \times \Gamma_{V_p}\times \{W\},\widetilde \Lca_{BA}(W))$$
is the induced morphism by $\tilde \beta_W^{\ast}$. That is, choose it in such a way that:
$$\widetilde {\psi_W}_{|x_0^{(i)} \times \Gamma_{V_p}}=v_i^{-1}$$
where $x_0^{(i)}$ is the origin of $\widehat X_{V_i}$ and $v^{-1}$ is $v_i^{-1}$ in the component $\widehat X_{V_i}$.
\end{defin}

As in the previous section, the bundle:
$$\Lca_{BA}(W):=\widetilde \Lca_{BA}(W)_{|\widetilde X_V \times \Gamma_{V_p}^- \times \{W\}}$$
is trivial over $\widetilde X_V \times \Gamma_{V_p}^-$, so, choosing a trivialization, the BA-function of a point $W\in \Gr(V_p)$ is defined by the following formula:
$$\psi_W(T_{\bullet},g_{\bullet})=v^{-1}\frac{\tau_W(g_{\bullet}\cdot \phi_1(T_{\bullet}))}{\tau_W(g_{\bullet})}$$
where $T_{\bullet}=(T_1,\dots ,T_r)$ and $g_{\bullet} \in \Gamma_{V_p}^-$.

In order to give an explicit expression of this function in characteristic zero, we compose the Abel morphism with the exponential isomorphism:
$$\widehat X_V \stackrel{\phi_1}{\rightarrow} \Jc(\widehat X_V) \stackrel{exp^{-1}}{\rightarrow} \widehat \As_{\infty}^r$$
that sends $T_j$ to the point in $\widehat \As_{\infty}^r$ with coordinates:
$$[T_j]:=\big((0,\dots ),\dots ,(T_j,\frac{T_j^2}{2},\frac{T_j^3}{3},\dots ),\dots ,(0,\dots )\big)$$
or, equivalentely, this map is induced by the ring hommorphism:
\begin{align*}
 k\{\{t_1^{(1)}, \dots \}\}\hat \otimes \cdots \hat \otimes k\{\{t_1^{(r)}, \dots \}\} &\to k[[T_1]] \times \cdots \times k[[T_r]]\\
t_i^{(j)}& \mapsto (0,\dots ,0,\frac{T_j^{(i)}}{i},0,\dots ,0)\,.
\end{align*}
It is clear that we have an addition map:
\begin{align*}
 \widehat X_V \times \Jc(\widehat X_V) & \to \Jc(\widehat X_V) \\
(T_{\bullet},t_{\bullet}) & \mapsto t_{\bullet}+[T_{\bullet}]
\end{align*}
where $T_{\bullet}=(T_1,\dots ,T_r)$, $t_{\bullet}=(t^{(1)},\dots ,t^{(r)})$ y $t_{\bullet}+[T_{\bullet}]$ denotes the point in $\widehat \As_{\infty}^r$ with coordinates $(\dots ,t_i^{(j)}+\frac{T_j^i}{i},\dots )$.

In this fashion, the BA-function of a point $W\in \Gr(V_p)$ (in characteristic zero) is defined by (\cite{MP1,MP4}):
$$\psi_W(T_{\bullet},t_{\bullet})=\prod_{i=1}^r \expo(-\sum_{j\geq 1} \frac{t_j^{(i)}}{T_i^j})\cdot \big( \frac{\tau_W(t_{\bullet}+[T_{\bullet}])}{\tau_W(t_{\bullet})}\big)$$

\begin{remark}\label{eq:r:BAOmega}
 If $\Omega \in \Gr(k((z)))$, its BA-function in characteristic zero is (\cite{AMP,MP1}):
$$\psi_{\Omega}(z,t)=\expo(-\sum_{j\geq 1} \frac{t_j}{z^j})\cdot \big( \frac{\tau_{\Omega}(t+[z])}{\tau_{\Omega}(t)}\big)$$
\end{remark}

\begin{remark}
In the case of points in the Grassmannian of $k((z))$, the Tau-function of a point $\Omega \in \Gr(k((z)))$ generates the subspace $\Omega$(\cite{MP1}), nevertheless, this is no longer true for a point $W\in \Gr(V_p)$. In order to solve this problem the following definition is used in \cite{MP4}.
\end{remark}

Recall the descomposition $V_p\iso V_1 \times \cdots \times V_r$. 
\begin{defin}\label{eq:d:uBA}
 The $u$-th BA-function of a point $W\in \Gr(V_p)$ is the function
{\small$$\psi_{u,W}(T_{\bullet},t_{\bullet}):=\big(\xi_{u1}\expo(-\sum_{j\geq 1} \frac{t_j^{(1)}}{T_1^j}) \frac{\tau_{W_{u1}}(t_{\bullet}+[T_1])}{\tau_W(t_{\bullet})},\dots ,\xi_{ur}\expo(-\sum_{j\geq 1} \frac{t_j^{(n)}}{T_r^j}) \frac{\tau_{W_{ur}}(t_{\bullet}+[T_r])}{\tau_W(t_{\bullet})}\big)$$}
where 
\begin{itemize}
 \item $1\leq u \leq r$.
 \item $W_{uv}:=(1,\dots ,T_u, \dots ,T_v^{-1},\dots ,1)\cdot W$.
 \item $t_{\bullet}+[T_v]:=(t^{(1)},\dots , t^{(v)}+[T_v],\dots ,t^{(n)})$.
 \item $\xi_{ui}$ is $-1$ if $i>u$ and $1$ if $i\leq u$.
\end{itemize}
\end{defin}

Now, let's see that $W$ can be generated (as an infinite subspace of $V_p$) by these $u$-th BA-functions.

Let $\widehat X_V^N$ be the formal scheme:
$$\widehat X_V^N :=\Spf \big((V_1^+)^{\otimes N}\big)\times \cdots \times \Spf \big((V_r^+)^{\otimes N}\big)$$
Denote $(V_i^+)^{\otimes N}=k[[T_i]]^{\otimes N}$ by $k[[x_1^{(i)},\dots ,x_N^{(i)}]]$. We have:
$$\Oc_{\widehat X_V^N}=k[[x_1^{(1)},\dots ,x_N^{(1)},\dots ,x_1^{(r)},\dots ,x_N^{(r)}]]\,.$$
The Abel morphism of degree $N$:
$$\phi_N \colon \widehat X_V^N \to \Jc(\widehat X_V)$$
 is defined by:
$$\prod_{k=1}^N (1-\frac{x_k^{(1)}}{T_1})^{-1},\dots ,\prod_{k=1}^N (1-\frac{x_k^{(r)}}{T_r})^{-1}\,.$$

\begin{lemma}\cite{MP4}
 Let $W\in \Gr^0(V_p)$ be a rational point. Let $N>0$ be an integer number such that $V_p/V_p+T_{\bullet}^N W=(0)$. Let $$\{f_i:=(f_i^{(1)}(T_1), \dots ,f_i^{(r)}(T_r))\,|\, 1\leq i \leq N\cdot r\}$$
be a basis for $V_p^+ \cap T_{\bullet}^N W$ as a $k$-vector space where $f_i^{(j)}\in V_j^+$.

Then:
{\tiny $$\phi_N^{\ast}\tau_W=\prod_{1\leq k<l\leq N,1\leq i\leq r}(x_l^{(i)}-x_k^{(i)})^{-1}\left| \begin{array}{ccccccc}
f_1^{(1)}(x_1^{(1)}) & \cdots & f_1^{(r)}(x_1^{(r)}) & \cdots & f_1^{(1)}(x_N^{(1)}) & \cdots & f_1^{(r)}(x_N^{(r)})\\
\vdots & & \vdots & & \vdots & & \vdots \\
f_{Nr}^{(1)}(x_1^{(1)}) & \cdots & f_{Nr}^{(r)}(x_1^{(r)}) & \cdots & f_{Nr}^{(1)}(x_N^{(1)}) & \cdots & f_{Nr}^{(r)}(x_N^{(r)})\\
\end{array} \right|$$}
as functions of $\Oc_{\widehat X_V^N}$ (up to non-zero constants).
\end{lemma}

\begin{remark}
 Recall that the superindex $0$ of $\Gr^0(V_p)$ denotes that we are taking the connected component of the Grassmannian where the index function of $W$ takes the value $0$ (see \cite[Def.3.3]{AMP}).
\end{remark}

\begin{thm}\cite{MP4} Take $W\in \Gr^0(V_p)$. Then:
$$\psi_{u,W}(T_{\bullet},t_{\bullet})=(1,\dots ,T_u, \dots ,1)\cdot \sum_{i>0}\big( \psi_{u,W}^{i,1}(T_1),\dots , \psi_{u,W}^{i,r}(T_r)\big)p_{ui,W}(t_{\bullet})$$
where
$$\{\big( \psi_{u,W}^{i,1}(T_1),\dots , \psi_{u,W}^{i,r}(T_r)\big)\,|\,i>0,1\leq u \leq n\}$$
is a basis for $W$ and $p_{ui,W}(t_{\bullet})$ are functions in $t_{\bullet}$.
\end{thm}

\begin{remark}\label{higgs:r:v_m}
For studing the index $m$ connected component ($m>0$), $\Gr^m(V_p)$, we need to choose an element $v_m$ such that $\di_k V_p^+/v_mV_p^+ =m$ in order to get an isomorphism $\Gr^m(V_p)\iso \Gr^0(V_p)$ that allow us to define the global  section $\Omega_+^m$ of the dual of the determinant bundle of $\Gr^m(V_p)$ as the image of the canonical section $\Omega_+$ of $\Det_{V_p}^{\ast}$ (see \cite{AMP} and \cite[Remark 4]{MP1}). Let us set $v_m$ as follows:
\begin{itemize}
 \item for $m\leq \frac{1}{2}(r-n)$, let $p,q,s,t$ be integer numbers defined by $-m=q\cdot (n-r)+ p$, $0\leq p<n-r$, $p=s\cdot r+t$, $0\leq t<r$. Therefore, we set
$$v_m:=(z^{-1}\cdot T_{\bullet})^q T_1^{s+1}\cdots T_t^{s+1}T_{t+1}^s \cdots T_r^s$$
 \item for $m> \frac{1}{2}(r-n)$, we set:
$$v_m:=(z^{-1}\cdot T_{\bullet})\cdot v_{r-n-m}^{-1}$$
\end{itemize}
\end{remark}

\begin{thm}\label{eq:t:BAgenW^m}\cite{MP4}
If $W\in \Gr^m(V_p)$ we have:
{\small$$\psi_{u,W}(T_{\bullet},t_{\bullet})=v_m^{-1}(1,\dots ,T_u, \dots ,1)\cdot \sum_{i>0}\big( \psi_{u,W}^{(i,1)}(T_1),\dots , \psi_{u,W}^{(i,r)}(T_r)\big)p_{ui,W}(t_{\bullet})$$}
where
$$\{\big( \psi_{u,W}^{(i,1)}(T_1),\dots , \psi_{u,W}^{(i,r)}(T_t)\big)\,|\,i>0,1\leq u \leq t\}$$
is a basis for  $W$ and $p_{ui,W}(t_{\bullet})$ are functions in $t_{\bullet}$.

In particular, an element of $V_p$ lies in $W$ if and only if it can be expressed as a linear combination of
$$\psi_{1,W}(T_{\bullet},t_{\bullet}),\dots ,\psi_{r,W}(T_{\bullet},t_{\bullet})$$
for certain values of the parameters $t_{\bullet}$.
\end{thm}

\begin{remark}\label{eq:BAgenOmega^s}
If $\Omega \in \Gr^{s}(k((z)))$ then (\cite{MP1}):
$$\psi_{\Omega}(z,t)=z^{1-s}\sum_{i>0}\psi_{\Omega}^{(i)}(z)p_i(t)$$
where $\{\psi_{\Omega}^{(i)}(z)\}_i$ is a basis for  $\Omega$ and $p_i(t)$ are functions in $t$.
In particular, if $\Omega$ is the point defined by the canonical line bundle $\omega_X$, then $s=g-1$, where $g$ is the genus of $X$. If we denote $\Omega^{-1}$ as the point defined by $\omega_X^{-1}$, then $s=1-3g$.
\end{remark}

\subsubsection{Adjoint Baker function.}\label{eq:ss:BAad}

Since $V_p$ is a finite separable $k((z))$-algebra, it carries the metric of the trace
$$\Tr \colon V_p \times V_p \to k((z))\,,$$
which is non-degenerated. Therefore, $V_p$ can be endowed with the non-degenerated pairing:
\begin{align*}
 \Tg_2 \colon V_p \times V_p & \to k \\
(a,b) & \mapsto \Res_{z=0}(\Tr(a,b))dz
\end{align*}

\begin{lemma}\cite{MP1,MP4}
 The pairing $\Tg_2$ induces an isomorphism of $k$-schemes:
\begin{align*}
 R \colon \Gr(V_p) & \to \Gr(V_p) \\
W &\mapsto W^{\bot} \,,
\end{align*}
where $W^{\bot}$ is writen for the orthogonal of $W$ w.r.t. $\Tg_2$.
\end{lemma}

We have the following properties (\cite{MP4}):
\begin{itemize}
 \item $R(\Gr^m(V_p))=\Gr^{-m}(V_p)$.
 \item $R^{\ast}\Det_{V_p}\iso \Det_{V_p}$.
 \item $(g\cdot W)^{\bot}=g^{-1}\cdot W^{\bot}$, for $W\in \Gr(V_p)$ and $g\in \Jc(\widehat X_V)$.
 \item $R^{\ast}\Omega_+^m=\Omega_+^{-m}$.
 \item $\tau_{W^{\bot}}(g)=\tau_{W}(g^{-1})$.
 \item $W_{uv}^{\bot}=W_{vu}$
\end{itemize}

\begin{defin}
The $u$-th Baker-Akhiezer function of a point $W\in \Gr(V_p)$ is:
$$\psi_{u,W}^{\ast}(T_{\bullet},t_{\bullet}):=\psi_{u,W^{\bot}}(T_{\bullet},-t_{\bullet})$$
In characteristic zero:
{\small\begin{align*}
 &\psi_{u,W}^{\ast}(T_{\bullet},t_{\bullet}):=\\
&=\big(\xi_{u1}\expo(\sum_{j\geq 1} \frac{t_j^{(1)}}{T_1^j}) \frac{\tau_{W_{u1}^{\bot}}(-t_{\bullet}+[T_1])}{\tau_W^{\bot}(-t_{\bullet})},\dots ,\xi_{un}\expo(\sum_{j\geq 1} \frac{t_j^{(n)}}{T_n^j}) \frac{\tau_{W_{ur}^{\bot}}(-t_{\bullet}+[T_r])}{\tau_W^{\bot}(-t_{\bullet})}\big)=\\
& =\big(\xi_{u1}\expo(\sum_{j\geq 1} \frac{t_j^{(1)}}{T_1^j}) \frac{\tau_{W_{1u}}(t_{\bullet}-[T_1])}{\tau_W(t_{\bullet})},\dots ,\xi_{ur}\expo(\sum_{j\geq 1} \frac{t_j^{(n)}}{T_r^j}) \frac{\tau_{W_{ru}}(t_{\bullet}-[T_r])}{\tau_W(t_{\bullet})}\big)
\end{align*}}
\end{defin}

\begin{thm}\label{eq:t:adBAgenW^m}
Take $W\in \Gr^m(V_p)$. We have:
{\small$$\psi_{u,W}^{\ast}(T_{\bullet},t_{\bullet})=v_{r-n-m}^{-1}(1,\dots ,T_u, \dots ,1)\cdot \sum_{i>0}\big( \psi_{u,W}^{\ast (i,1)}(T_1),\dots , \psi_{u,W}^{\ast (i,r)}(T_r)\big)p_{ui,W}^{\ast}(t_{\bullet})$$}
where
$$\{\big( \psi_{u,W}^{\ast (i,1)}(T_1),\dots , \psi_{u,W}^{\ast (i,r)}(T_r)\big)\,|\,i>0,1\leq u \leq r\}$$
is a basis for $W^{\bot}$ and $p_{ui,W}^{\ast}(t_{\bullet})$ are functions in $t_{\bullet}$.

In particular, an element of $V_p$ lies in $W^{\bot}$ if and only if it can be expressed as a linear combination of
$$\psi_{1,W}^{\ast}(T_{\bullet},t_{\bullet}),\dots ,\psi_{r,W}^{\ast}(T_{\bullet},t_{\bullet})$$
for certain values of the parameters $t_{\bullet}$.
\end{thm}

\subsection{Product of BA-functions.}\label{eq:ss:prodBA}\qquad

Take $\Omega \in \Gr(k((z)))$ and $W\in \Gr(V_p)$. Consider its BA-functions:
$$\psi_{\Omega}(z,t) \qquad \quad \psi_{W}(T_{\bullet},t_{\bullet})$$
(that lies over $\widehat X \times \Gamma^-$ and $\widehat X_V \times \Gamma_{V_p}^-$ respectively).

The graph, $\widehat X_V \to \widehat X_V \times \widehat X$, of the morphism $\widehat \pi \colon \widehat X_V \to \widehat X$ induces:
$$f \colon \widehat X_V \times \Gamma_{V_p}^- \times \Gamma^- \to (\widehat X_V \times \Gamma_{V_p}^-) \times (\widehat X \times \Gamma^-)\,.$$

\begin{remark}
Note that the graph is the morphism induced by:
\begin{align*}
 k[[z]] \otimes_k (k[[T_1]]\times \cdots \times k[[T_r]]) & \to k[[T_1]]\times \cdots \times k[[T_r]] \\
z\otimes_k (T_1,\dots ,T_r) & \mapsto (zT_1,\dots ,zT_r)
\end{align*}
and this expression makes sense because we know how to express $z$ in terms of $T_i$ by equation (\ref{eq:e:descom}).
\end{remark}

\begin{defin}
The product of the BA-functions of $\Omega \in \Gr(k((z)))$ and $W\in \Gr(V_p)$ is defined by:
$$\psi_{W}(T_{\bullet},t_{\bullet}) \ast \psi_{\Omega}(z,t):=f^{\ast}\big( \psi_{W}(T_{\bullet},t_{\bullet}) \otimes_k \psi_{\Omega}(z,t) \big)$$
The product of $\psi_{\Omega}(z,t)$ by the $u$-th BA-function of $W$, $\psi_{u,W}(T_{\bullet},t_{\bullet})$, can be defined in similar way.
\end{defin}

\begin{prop}\label{eq:p:genOmegaW}
Let $W\in \Gr(V_p)$ and $\Omega \in \Gr(k((z)))$ be such that $\Omega \cdot W \in \Gr(V_p)$. Then
$$\{\psi_{\Omega}(z,t)\ast \psi_{1,W}(T_{\bullet},t_{\bullet}),\dots ,\psi_{\Omega}(z,t)\ast \psi_{r,W}(T_{\bullet},t_{\bullet})\}$$
is a generating system for $\Omega \cdot W$. If the characteristic of $k$ is zero, then:
$$\psi_{\Omega}(z,t)\ast \psi_{u,W}(T_{\bullet},t_{\bullet})=\big(\psi_{\Omega}(z,t)\cdot \psi_{u,W}^{(1)}(T_1,t_{\bullet}),\dots ,\psi_{\Omega}(z,t)\cdot \psi_{u,W}^{(r)}(T_r,t_{\bullet})\big)$$
where 
$$\psi_{\Omega}(z,t)=\expo(-\sum_{j\geq 1} \frac{t_j}{z^j})\cdot \big( \frac{\tau_{\Omega}(t+[z])}{\tau_{\Omega}(t)}\big)$$
and
$$\psi_{u,W}^{(i)}=\xi_{ui}\expo(-\sum_{j\geq 1} \frac{t_j^{(i)}}{T_i^j}) \frac{\tau_{W_{ui}}(t_{\bullet}+[T_i])}{\tau_W(t_{\bullet})}$$
\end{prop}
\begin{proof}
They are a generating system by \ref{eq:t:BAgenW^m} and \ref{eq:BAgenOmega^s}. Explicit formula in characteristic zero follows from \ref{eq:r:BAOmega}, \ref{eq:d:uBA} and the definition of the graph morphism.
\end{proof}

\subsection{Equations of $\Higgs_X^{\infty}$.}\qquad

Theorem \ref{higgs:t:carHiggs} says that a rational point $(W,p(T)) \in \U_X^{\infty}(k) \times \As(k)$ lies on the image of 
$$\Higgs_X^{\infty}(k) \hookrightarrow \U_X^{\infty}(k) \times \As(k)$$
if and only if $T(W) \subseteq W\cdot \Omega$, where $\Omega \in \Gr(k((z)))(k)$ is the point defined by the canonical line bundle $\omega_X$ and $T$ is the homotety multiplying by $T$ in $V_p$. Bearing in mind the descomposition $V_p\iso V_1 \times \cdots \times V_r$, $T$ translates into the homotety multiplying by $T_{\bullet}=(T_1, \dots ,T_r)$, therefore, if we think of $T$ as an operator in $V_p$ is easy to see that $T$ is self-adjoint w.r.t. the pairing $\Tg_2$ defined in section \ref{eq:ss:BAad}, that is:
$$\Tg_2(T\cdot v,v')=\Tg_2(v,T\cdot v') \qquad \forall v,v'\in V_p\iso V_1\times \cdots V_r\,.$$

%que como endomorfismo de $k((z))$-espacios vectoriales de $k((z))^n$ tiene por matriz:
%$$\left( \begin{array}{cccccc}
%0 & & & & 0 & -a_n\\
%1 & \cdot & & & \cdot & -a_{n-1} \\
%\cdot & 1 &\cdot  & & \cdot & \cdot \\
%\cdot & & \cdot & \cdot & \cdot & \cdot \\
%\cdot & & & \cdot & 0 & -a_2 \\
%0 & \cdot & & 0 & 1 & -a_1
%\end{array} \right)\,.$$
%Observemos que $T$ puede interpretarse como un elemento en $\Gamma_{V_p}$ y por lo tanto podemos reescribir la %relaci\'on $T(W) \subseteq W\cdot \Omega$ como:
%$$\Omega^{-1} \cdot W \subseteq T^{-1}\cdot W \,.$$

\begin{thm}\label{eq:t:eqHiggs_X}
 Let $(W,p(T)) \in \U_X^{\infty}(k) \times \As(k)$ be a rational point, $\Omega \in \Gr^{g-1}(k((z)))$ and assume that  $W\in \Gr^m(V_p)$ (with $m\neq \frac{1}{2}(r-n)$).

$(W,p(T))$ lies in the image of $\Higgs_X^{\infty}(k) \hookrightarrow \U_X^{\infty}(k) \times \As(k)$ if and only if:
$$\Tg_2 (\frac{v_{r-n-m}\cdot T_{\bullet}\cdot \psi_{u,W}^{\ast}(T_{\bullet},t_{\bullet})}{(1,\dots ,T_u,\dots ,1)}, \frac{v_m\cdot \psi_{\Omega^{-1}}(z,t)\cdot \psi_{v,W}(T_{\bullet},t_{\bullet})}{z^{3g}\cdot (1,\dots ,T_v,\dots 1)})=0$$
for all $1\leq u,v \leq r$.
\end{thm}

\begin{proof}
By theorem \ref{higgs:t:carHiggs}, $(W,p(T)) \in \U_X^{\infty}(k) \times \As(k)$ lies on the image of  $\Higgs_X^{\infty}(k) \hookrightarrow \U_X^{\infty}(k) \times \As(k)$ if and only if $T\cdot W \subseteq W\cdot \Omega$. Now:
$$T\cdot W \subseteq W\cdot \Omega \iff \Omega^{-1} \cdot W \subseteq T^{-1}\cdot W \iff (T^{-1}\cdot W)^{\bot} \subseteq (\Omega^{-1} \cdot W)^{\bot}$$
where $\bot$ denotes the orthogonal w.r.t. $\Tg_2$. 

Because $T$ is self-adjoint, we have:
\begin{align*}
(T^{-1}\cdot W)^{\bot}:&=\{v\in V_p\,|\, \Tg_2(v,T^{-1}\cdot w)=0 \quad \forall w\in W\}=\\
&=\{v\in V_p\,|\, \Tg_2(T^{-1} \cdot v,w)=0 \quad \forall w\in W\}=\\
&=\{v\in V_p \,|\, T^{-1}\cdot v \in W^{\bot}\}=T\cdot W^{\bot}
\end{align*}
Since multiplying by $T$ corresponds to multiplying by $T_{\bullet}$, using the theorem \ref{eq:t:adBAgenW^m} we have that:
$$\psi_{1,W}^{\ast}(T_{\bullet},t_{\bullet}),\dots ,\psi_{r,W}^{\ast}(T_{\bullet},t_{\bullet})$$
is a generating system for $W^{\bot}$. Therefore
$$T_{\bullet}\cdot \psi_{1,W}^{\ast}(T_{\bullet},t_{\bullet}),\dots ,T_{\bullet}\cdot \psi_{r,W}^{\ast}(T_{\bullet},t_{\bullet})$$
is a generating system for $T\cdot W^{\bot}$ and (by theorem \ref{eq:t:adBAgenW^m}):
\begin{align*}
 &T_{\bullet}\cdot \psi_{u,W}^{\ast}(T_{\bullet},t_{\bullet})=\\
&=v_{r-n-m}^{-1}(1,\dots ,T_u, \dots ,1)\cdot \sum_{i>0}\big( T_1\cdot \psi_{u,W}^{\ast (i,1)}(T_1),\dots , T_r\cdot \psi_{u,W}^{\ast (i,r)}(T_r)\big)p_{ui,W}^{\ast}(t_{\bullet})
\end{align*}

where
$$\{\big( T_1\cdot \psi_{u,W}^{\ast (i,1)}(T_1),\dots , T_r\cdot \psi_{u,W}^{\ast (i,r)}(T_r)\big)\,|\,i>0,1\leq u \leq r\}$$
is a basis for $T\cdot W^{\bot}$ and $p_{ui,W}^{\ast}(t_{\bullet})$ are functions in $t_{\bullet}$.

Since $\Tg_2$ is non-degenerated and because of the expressions \ref{eq:BAgenOmega^s}, \ref{eq:t:BAgenW^m}, \ref{eq:t:adBAgenW^m} of the functions $\psi_{\Omega^{-1}}(z,t)$, $\psi_{v,W}(T_{\bullet},t_{\bullet})$ and $\psi_{u,W}^{\ast}(T_{\bullet},t_{\bullet})$,  we conclude that $T\cdot W^{\bot} \subseteq (\Omega^{-1} \cdot W)^{\bot}$ if and only if the following set of equations is satisfied:
$$\Tg_2 (\frac{v_{r-n-m}\cdot T_{\bullet}\cdot \psi_{u,W}^{\ast}(T_{\bullet},t_{\bullet})}{(1,\dots ,T_u,\dots ,1)}, \frac{v_m\cdot \psi_{\Omega^{-1}}(z,t)\cdot \psi_{v,W}(T_{\bullet},t_{\bullet})}{z^{3g}\cdot (1,\dots ,T_v,\dots 1)})=0$$
for all $1\leq u,v \leq r$.
\end{proof}

%\begin{corol}
%$(W,p(T))$ lies in the image of $\Higgs_X^{\infty}(k) \hookrightarrow \U_X^{\infty}(k) \times \As(k)$ if and only %if:
%$$\sum_{i=1}^r \Res_{z=0}\big(\sum_{j=1}^{n_i}(\xi_i^jT_i)^{2-\delta_{iu}-\delta_{iv}}\psi_{\Omega^{-1}}(z,t)\cdot %\psi_{u,W}^{\ast,(i)}(\xi_i^jT_i,t_{\bullet})\cdot %\psi_{v,W}^{(i)}(\xi_i^jT_i,t_{\bullet})\big)\frac{dz}{z^{3g+1}}=0$$
%for all $1\leq u,v \leq r$.
%\end{corol}

%\begin{proof}
% We only have to bear in mind that $v_m\cdot v_{r-n-m}=z^{-1}\cdot T_{\bullet}$ (see remark \ref{higgs:r:v_m}), the %definition of $\Tg_2$ (section \ref{eq:ss:BAad}) and that the trace morphism is described by:
%\begin{align*}
% \Tr \colon V_1 \times \cdots \times V_r & \to k((z)) \\
%\big(f_1(T_1),\dots ,f_r(T_r) \big) & \mapsto \sum_{j=1}^{n_1}f_1(\xi_1^jT_1)+ \cdots +\sum_{j=1}^{n_r}f_r(\xi_r^j %T_r)
%\end{align*}
%where $\xi_i$ are $n_i$-th roots of unity.
%\end{proof}

\subsection{Equations for irreducible polynomial.}\quad

In this section equations describing $\Higgs_X^{\infty}$ for the case in which the spectral cover $\pi \colon X_{\varphi} \to X$ is totally ramified at $x\in X$ (that is, $\pi^{-1}(x)=ny$) are computed explicitely in terms of the characteristic coefficients of the formal Higgs field.

In our formal setting, to assume this hypothesis implies that the characteristic polynomial of the formal Higgs field is irreducible. So, from now on, assume we are given an irreducible polynomial
$$p(T)=T^n-a_1T^{n-1}+\cdots +(-1)^na_n$$

The multiplycation by $T$ in $V_p=k((z))[T]/p(T)$ w.r.t. the basis $\{1,T,\dots ,T^{n-1}\}$ corresponds to the matrix:
$$T=\left( \begin{array}{cccccc}
0 & & & & 0 & (-1)^{n+1}a_n\\
1 & \cdot & & & \cdot & (-1)^na_{n-1} \\
\cdot & 1 &\cdot  & & \cdot & \cdot \\
\cdot & & \cdot & \cdot & \cdot & \cdot \\
\cdot & & & \cdot & 0 & -a_2 \\
0 & \cdot & & 0 & 1 & a_1
\end{array} \right)$$

By equation \ref{eq:e:descom} we are in the following case:
$$V_p=k((z))[T]/p(T)\iso k((z))[T]/T^n-zu(T)\iso k((T))$$
so that $u(T)$ can be computed explicitely in terms of $a_1,\dots ,a_n$. Note in addition that multiplying by $T^{-1}$ makes sense in $V_p$ (but not in $V_p^+$).

On one hand, if $W\in \Gr^m(V_p)$ then its BA-function is given by: 
\begin{equation}\label{eq:e:Psi_W}
\psi_W(T,t_n)=T^{1-m}\sum_{i>0}\psi_W^{(i)}(T)p_i(t_n)
\end{equation}
where $\{\psi_W^{(i)}(T)\}_i$ is a basis for  $W$ and $p_i(t_n)$ are functions in $t_n$ (see remark \ref{eq:BAgenOmega^s} and notice that we are writing $t_n$ for the coordinates of $\Gamma_{V_p}$). Observe that the orthogonal $W^{\bot}$ w.r.t. the metric $\Tg_2$ definied in section \ref{eq:ss:BAad} lies on $\Gr^{-m}(V_p)$, therefore:
\begin{equation}\label{eq:e:Psi_Wadj}
\psi_W^{\ast}(T,t_n)=T^{1+m}\sum_{i>0}\psi_W^{\ast,(i)}(T)p_i^{\ast}(t_n)
\end{equation}
where $\{\psi_W^{\ast,(i)}(T)\}_i$ is a basis for  $W^{\bot}$ and $p_i^{\ast}(t_n)$ are functions in $t_n$. 

On the other hand, as in section \ref{eq:ss:prodBA} we can define the product of the BA-function $\Psi_{\Omega^{-1}}(z,t)$ of $\Omega^{-1}\in \Gr(k((z)))$ by the BA-function $\Psi_W(T,t_n)$ of $W\in \Gr(V_p)$ bearing in mind that the $k((z))$-algebra structure of $V_p\iso k((T))$ is given by:
\begin{align*}
 k((z))&\to k((T))\\
z &\mapsto T^nu(T)^{-1}
\end{align*}
and the graph is the morphism induced by:
\begin{align*}
 k[[z]] \otimes_k k[[T]] & \to k[[T]]\\
z\otimes_k T & \mapsto zT
\end{align*}

In order to compute explicit equations for $\Higgs_X^{\infty}$, BA-functions will be thought as mere functions on $V_p$, that is, as polynomials on $T$ with coefficients in $k((z))$. In this fashion, we set:
\begin{equation}\label{eq:e:Psi_Wpol}
 \Psi_W=\sum_{i=0}^{n-1}\Psi_W(i)T^i\,, \qquad \Psi_W(i)\in k((z))
\end{equation}
\begin{equation}\label{eq:e:Psi_Wadjpol}
 \Psi_W^{\ast}=\sum_{i=0}^{n-1}\Psi_W^{\ast}(i)T^i\,, \qquad \Psi_W^{\ast}(i)\in k((z))
\end{equation}

Let $(E,\varphi,\phi)$ be a rational point of $\Higgs_X^{\infty}$ and 
$$p_{\widehat \varphi}(T)=T^n-a_1T^{n-1}+\cdots +(-1)^na_n\in \As(k)$$
its formal characteristic polynomial. Newton-Girard identities (see \cite{McD}) express $\Tr(T^k)$ in terms of $a_i=\Tr(\Lambda^i T)$ as follows:
\begin{equation}\label{eq:e:newton}
\Tr(T^{k})=\left | \begin{array}{cccccc}
a_1 & 1 & 0 & \cdot & \cdot & \cdot \\
2a_2 & a_1 & 1 & 0 & \cdot & \cdot  \\
\cdot & \cdot &\cdot  & \cdot & \cdot & \cdot \\
\cdot & \cdot &  & \cdot &  &  \\
\cdot & \cdot &  &  & \cdot &  \\
ka_k & a_{k-1} & \cdot & \cdot & \cdot & e_1
\end{array} \right |\in k((z)) \,, \qquad \forall k\geq 0
\end{equation}
where $a_k=0$ for $k>n$.

Consider now the following stratification of $\Higgs_X^{\infty}$ into connected components:
$$\Higgs_X^{\infty}=\coprod_{\underline n}\Higgs_X^{\infty}(\underline n)$$
where $\Higgs_X^{\infty}(\underline n)$ denotes the connected component of $\Higgs_X^{\infty}$ for which the spectral cover $\pi \colon X_{\varphi} \to X$ is ramified with partition numbers $\underline n=(n_1,\dots ,n_r)$ (where $n_1+\cdots +n_r=n$), that is to say, $\pi^{-1}(x)=n_1y_1+\cdots +n_ry_r$. Note that this stratification is induced by the decomposition of the characteristic polynomial into irreducible factors.

We will write $\Higgs_X^{\infty}(n)$ for the connected componet on which the spectral cover is totally ramified ($\pi^{-1}(x)=ny$).

\begin{thm}\label{eq:t:eq_tot_rami}
 Let $(W,p(T))$ be a rational point of $\U_X^{\infty}\times \As$. Then, $(W,p(T))$ lies on the image of $\Higgs_X^{\infty}(n) \hookrightarrow \U_X^{\infty} \times \As$ if and only if $p(T)$ is irreducible and:
$$\sum_{k=0}^{2n-2}\sum_{i+j=k}\Res_{z=0}\big(\Psi_W^{\ast}(i)\Psi_{\Omega^{-1}}\Psi_W(j)\Tr(T^{k-1}) \big)\frac{dz}{z^{3g}}=0$$
\end{thm}

\begin{proof}
Using the expressions \ref{eq:e:Psi_W}, \ref{eq:e:Psi_Wadj}, remark \ref{eq:BAgenOmega^s} and theorem \ref{eq:t:eqHiggs_X} we get that $(W,p(T))$ lies in the image of $\Higgs_X^{\infty,n}(k) \hookrightarrow \U_X^{\infty}(k) \times \As(k)$ if and only if:
\begin{equation}\label{eq:e:T_2_tot_ram}
 \Tg_2 \big(T\cdot T^{-1+m} \cdot \psi_W^{\ast}(T,t_n), T^{-1-m}\cdot \psi_{\Omega^{-1}}(z,t)\cdot \psi_W(T,t_n)\cdot z^{-3g}\big)=0
\end{equation}
In order to make this formula explicit, observe that the $k$-algebra structure in $V_p$ is nothing but the standard polynomial multiplication mod $p(T)$. So using the polynomial expressions \ref{eq:e:Psi_Wpol}, \ref{eq:e:Psi_Wadjpol} and the  definition of $\Tg_2$ (see section \ref{eq:ss:BAad}), the equation \ref{eq:e:T_2_tot_ram} becomes:
\begin{equation}
\Res_{z=0}\big(\Tr(\sum_{k=0}^{2n-2}\sum_{i+j=k}\Psi_W^{\ast}(i)\Psi_{\Omega^{-1}}\Psi_W(j)T^{k-1}z^{-3g})\big)dz=0
\end{equation}
Bearing in mind that the trace map is $k((z))$-linear and the linearity of $\Res$ we get the result.
\end{proof}

\begin{remark}
 These techniques can also be used to give a different way of computing the equations for the connected component $\Higgs_X^{\infty}(1,\dots ,1)$ (on which the spectral cover is not ramified) treated on \cite{LM1}.
\end{remark}

%Let $(E,\varphi, \phi)$ be a rational point of $\Higgs_X^{\infty}$ and let 
%$$p_{\widehat \varphi}(T)=T^n-a_1T^{n-1}+\cdots +(-1)^na_n \in \As(k)$$
%be the characteristic polynomial of the formal Higgs field $\widehat \varphi$.

\bibliographystyle{siam}
\bibliography{biblio}
\end{document}